\documentclass{imanum}
\usepackage{graphicx}
\usepackage{color}
\usepackage{cleveref}
\usepackage{algorithm}
\usepackage{diagbox} 
\usepackage{soul}
\usepackage{epstopdf}
\usepackage[noend]{algpseudocode}
\def\eu{\ensuremath{\mathrm{e}}}
\def\iu{\ensuremath{\mathrm{i}}}

\newcounter{def}
\newtheorem{defini}[def]{Definition}

\jno{drnxxx}
\received{}
\revised{}

\begin{document}
	
	\title{A novel spectral method for the semi-classical Schr\"odinger equation based on the Gaussian wave-packet transform}
	\shorttitle{A novel spectral method based on GWPT}
	
	\author{%
		{\sc
			Borui Miao\thanks{Email: mbr20@mails.tsinghua.edu.cn}
		} \\[2pt]
		Department of Mathematical Sciences, Tsinghua University, Beijing, 100084, China\\[6pt]
		{\sc and}\\[6pt]
		{\sc Giovanni Russo\thanks{Email: russo@dmi.unict.it}}\\[2pt]
		Department of Mathematics and Computer Science, \\University of Catania, Catania, 95125, Italy\\[6pt]
		{\sc and}\\[6pt]
		{\sc Zhennan Zhou\thanks{Corresponding author. Email: zhennan@bicmr.pku.edu.cn}}\\[2pt]
		Beijing International Center for Mathematical Research, \\Peking University, Beijing, 100871, China}
	\shortauthorlist{}
	
	\maketitle
	
	\begin{abstract}
		{In this article, we develop and analyse a new spectral method to solve the semi-classical Schr\"odinger equation based on the Gaussian wave-packet transform (GWPT) and  Hagedorn's semi-classical wave-packets (HWP). The GWPT equivalently recasts the highly oscillatory wave equation as a much less oscillatory one (the $w$ equation) coupled with a set of ordinary differential equations governing the dynamics of the so-called GWPT parameters. The Hamiltonian of the $ w $ equation consists of a quadratic part and a small non-quadratic perturbation, which is of order $ \mathcal{O}(\sqrt{\varepsilon
			}) $, where $ \varepsilon\ll 1 $ is the rescaled Planck's constant. By expanding the solution of the $ w $ equation as a superposition of Hagedorn's wave-packets, we construct a spectral method while the $ \mathcal{O}(\sqrt{\varepsilon}) $ perturbation part is treated by the Galerkin approximation. This numerical implementation of the GWPT avoids imposing artificial boundary conditions and facilitates rigorous numerical analysis. %
			For arbitrary dimensional cases, we establish how the error of solving the semi-classical Schr\"odinger equation with the GWPT is determined by the errors of solving the $ w $ equation and the GWPT parameters. We prove that this scheme has the spectral convergence with respect to the number of Hagedorn's wave-packets in one dimension. Extensive numerical tests are provided to demonstrate the properties of the proposed method. 
		}
		{semi-classical Schr\"odiner equation; Gaussian wave-packet transform; Hagedrn's wave-packets; spectral method.}
	\end{abstract}

	\section{Introduction}\label{sec:intro}
	We consider the time-dependent semi-classical Schr\"odinger equation that governs the dynamics of a quantum particle in an external potential:
	\begin{align}\label{eqn:SEqn}
		\iu\varepsilon\partial_t\psi(x,t) =\hat{H}(x,-\iu\varepsilon \nabla_x)\psi(x,t)&=-\frac{\varepsilon^2}{2}\Delta_{x}\psi(x,t)+V(x)\psi(x,t),\quad x\in \mathbb{R}^{d},t\in\mathbb{R}^{+}; \\
		\psi(x,0)&=\psi_{0}(x),\quad x\in \mathbb{R}^{d}.\label{eqn:Sini_val}
	\end{align}
	Here, $ \varepsilon\ll 1 $ denotes the rescaled Planck's constant, which is also referred to as the semi-classical parameter and $ \hat{H} $ denotes the quantum Hamiltonian operator obtained from the classical Hamiltonian symbol $ H(x,p) $ by the Weyl quantization. 
	$ V(x) $ in \eqref{eqn:SEqn} is a real-valued scalar potential function and $ \psi(x,t) $ is a complex-valued wave function. 
	This is one of the fundamental equations in quantum mechanics, and bridges the gap between realms of quantum mechanics and classical mechanics as $\varepsilon \rightarrow 0$. It also gives rise to many interesting problems in analysis and numerical computation, see e.g. \citet{Heller2018,Jin2011,Lasser2020,Zworski2012}. \par
	Due to the smallness of the semi-classical parameter $ \varepsilon $, the solutions $ \psi(x,t) $ is highly oscillatory in both space and time on a scale $ \mathcal{O}(\varepsilon) $ \citep[see][]{Jin2011}. 
	Also, many physical applications demand accurate simulations of the semi-classical Schr\"odinger equation \eqref{eqn:SEqn} in high dimensions. 
	Both facts lead to significant computational burdens if conventional algorithms 
	are directly applied to equation \eqref{eqn:SEqn}-\eqref{eqn:Sini_val}. 
	Comparatively speaking, the Fourier spectral method studied in \cite{Bao2002}. In their paper, it is most efficient among grid-based numerical methods, where the meshing strategy $ \Delta t = \mathcal{O}(\varepsilon) $ and $ \Delta x = \mathcal{O}(\varepsilon) $ is needed to approximate the wave function. If one only aims to capture the correct physical observables, the time step size can be relaxed to $ \mathcal{O}(1) $, while the spatial mesh constraint $ \Delta x = \mathcal{O}(\varepsilon) $ is still necessary (see \cite{Golse2020} for a rigorous justification). In recent years, this method has been extended to the semi-classical Schr\"odinger equations with vector potentials and time-dependent self-consistent field systems, see \cite{Jin2013, Ma2017, Jin2017}. This method can also be generalized to produce high order splitting methods, see \cite{Bader2014} for details. \par
	However, to implement the grid-based methods in high dimensions, one has to at least put a few collocation points for each wavelength on a scale $\mathcal{O}(\varepsilon)$ in every space dimension. This will cause unaffordable costs in numerical implementation. To overcome the curse of dimensionality, one can apply the splitting methods on a sparse grid \citep[see][]{Gradinaru2008} when $ \varepsilon\sim \mathcal{O}(1) $. Unfortunately, the authors pointed out in \cite{Lasser2020} that more than $ \mathcal{O}(\varepsilon^{-d}) $ points are needed in the semi-classical regime, making the Fourier spectral method on a sparse grid infeasible.\par  
	For simulating the semi-classical Sch\"odinger equation in high dimensions, many asymptotic approaches have been developed to avoid resolving the highly oscillatory wave function. For example, the Gaussian beam method  (see e.g. \cite{Liu2012,Jin2008}) reduces the quantum dynamics to Gaussian beam dynamics that is determined by a set of ordinary differential equations with $ \mathcal{O}(md^2) $ variables. Here $ m $ is the number of Gaussian beams which is determined by the initial condition. This method also brings an approximation error of order $ \mathcal{O}(\varepsilon^{1/2}) $, which is not satisfactory unless $ \varepsilon $ is extremely small. High order Gaussian beams (see \cite{Liu2012,Jin2011a}) can be used to reduce the error of order $ C_k(T)\varepsilon^{k/2} $, but this may not lead to higher accuracy for moderate $\varepsilon$ either. In particular, for a given $\varepsilon$, it is not guaranteed that increasing $k$ provides a decrease in the error. 
	Continuous superposition of Gaussian beams is investigated in \cite{Zheng2014}, and it is proved that this method have error of order $\mathcal{O}(\varepsilon)$ when the initial is of WKB type. But it is not clear how to develop high order algorithms under this framework.
	For other numerical methods to the semi-classical Schr\"odinger equation, the readers may refer to the recent reviews \cite{Jin2011,Lasser2020}.   
	\par  
	In \cite{Hagedorn1998}, Hagedorn constructed a set of semi-classical wave-packets that constitute a basis of $ \mathbb{L}^2(\mathbb{R}^d) $. When the Hamiltonian operator is quadratic (i.e.\ the Hamiltonian symbol $ H(x,p) $ is a second order polynomial in $ x,\,p $), the full quantum dynamics is equivalent to the dynamics of the wave packet parameters. The author also built up explicit dynamics for coefficients of the wave-packets that can be applied to approximate the solution with a quantifiable error estimate for general quantum Hamiltonian operators. However, the constructed dynamics is too complicated to be implemented as a numerical method. Here we comment that the dynamics of Hagedorn's semi-classical wave-packets can be regarded as an alternative formulation of the wave-packet dynamics.\par  
	Based on the semi-classical wave-packets, Faou et al.\ \citep[see][]{Faou2009} first turned them into a computational tool (later abbreviated as the HWP method). They use the Galerkin approximation to deal with the general potential, making it feasible for numerical implementation. They also observed that sparse grid techniques can be applied when choosing the superposition of wave-packets in high dimensions. Compared with Gaussian beam method, the HWP method improved the approximation error to $ C(T)\varepsilon^{k/2} $ \citep[see, e.g.,][]{Lasser2020}. This method can be generalized to produce higher order splitting methods, \citep[see ][]{Blanes2020,Gradinaru2013}. And it is rigorously proved the convergence in time in \cite{Gradinaru2013}. For the validity of the Galerkin approximation and the generalization of HWP to the semi-classical Schr\"odinger equations with vector potentials, see \cite{Zhou2014} for details.\par 
	Russo and Smereka proposed a new computational framework called the Gaussian wave-packet transform in \cite{Russo2013,Russo2014} (abbreviated as GWPT in later text) to solve the semi-classical Schr\"odinger equation in arbitrary spatial dimensions. It is a reversible transformation and reduces the semi-classical Schr\"odinger equation to the wave-packet dynamics together with the time evolution of a rescaled wave-function $ w $. The GWPT parameters that describe the wave-packet dynamics satisfy a set of ordinary differential equations. The rescaled wave function $ w $ satisfies a Schr\"odinger-type equation with a time-dependent Hamiltonian operator, and is less oscillatory than the original solution to the Schr\"odinger equation \eqref{eqn:SEqn}. Another fact about the $ w $ equation is that its quantum Hamiltonian consists of a quadratic part and a $ \mathcal{O}(\sqrt{\varepsilon}) $ perturbation part. This fact motivated us to treat the non-quadratic part as a perturbation. For extension of the GWPT to semi-classical Schr\"odinger equation with vector potential, see \cite{Zhou2019}.
	In all these works \citep[][]{Russo2013, Russo2014, Zhou2019} the GWPT formulation has been implemented as follows. For the $w$ equation, a second or fourth order splitting method is used in time and a Fourier-spectral discretization in a truncated spatial domain has been adopted, with periodic boundary conditions, which may be a reasonable approximation if the wave function decays fast enough. The GWPT parameters are numerically solved by a fourth order Runge-Kutta scheme. Thus, the whole scheme is spectrally accurate in space and second or fourth order accurate in time, given that the domain truncation is proper.
	\par  
	Now we are concerned with the following aspects of the GWPT. First, although this transformation can effectively reduce the computational burden, its related numerical analysis remains undone. Hereby, we ask how the error in solving the Gaussian wave-packet dynamics and the $ w $ equation affects the error of the reconstructed solution $ \psi $. Second, in \cite{Russo2013,Russo2014} the authors proposed to solve the $w$ equation with a Fourier spectral method. As we mentioned before, an artificial periodic boundary condition may cause difficulties for the numerical analysis of the method. 
	\par
	In this paper, we propose a spectral method based on the GWPT and the HWP method. Recall that the quantum Hamiltonian operator of the $ w $ equation is divided into a quadratic part and a $ \mathcal{O}(\sqrt{\varepsilon}) $ non-quadratic perturbation. We can expand the solution to the $w$ equation as superpositions of Hagedorn wave-packets, and thus the time evolution of the quadratic Hamiltonian is recast in terms of wave packet-dynamics. Then, we use the Galerkin approximation to deal with the $ \mathcal{O}(\sqrt{\varepsilon}) $ non-quadratic part. This numerical implementation of the GWPT avoids imposing artificial boundary conditions since it solves the $w$ equation on the whole space. \par 
	For the numerical analysis, we first develop a framework to estimate the error bound in the wave function $\psi$ for generic numerical methods which are based on the GWPT. Here, we only need to know the $ \mathbb{L}^{\infty} $ error when solving the GWPT parameters and discrete $ \mathbb{L}^2 $ error when solving the $ w $ equation to give such an error estimate. Second, we prove an asymptotic error bound with respect to the number of wave-packets rigorously in one dimension, showing the spectral accuracy. Compared with the previous results, which are at most algebraic accurate with respect to the number of wave-packets, our proposed scheme is certainly superior and appears promising in high dimensional simulations, since it greatly saves the computational cost.  
	Extensive numerical tests are provided to demonstrate the properties of the method and verify the above arguments.  \par 
	
	This article is organized as follows: In \Cref{sec:Review} we will briefly review the GWPT  and the HWP for the Schr\"odinger equation and propose a novel spectral method in solving the semi-classical Schrödinger equation  based on the GWPT. In \Cref{sec:Numerical_Analysis} we carry out the related numerical analysis. We provide systematic numerical tests in \Cref{sec:Numerical_Test}, demonstrating the spectral convergence and other properties of the proposed scheme. We give some concluding remarks in \Cref{sec:Remarks}.
	The appendix is devoted to explain some notation and technical details. 
	
	\section{Description of the Numerical Method}\label{sec:Review}
	In this section, we will present a novel spectral method for the semi-classical Schr\"odinger equation based on the GWPT and the HWP. In particular, the efficiency of the method is not affected by the smallness of the parameter $\varepsilon$, and the approximation error can be rigorously quantified. To describe the method in detail, we will first review some properties of the GWPT and HWP.
	\subsection{A Brief Review of the GWPT}\label{subsec:review_GWPT}
	For heuristic purposes, let's first consider the semi-classical Schr\"odinger equation \eqref{eqn:SEqn} with a normalized Gaussian initial value
	\begin{equation}\label{eqn:ini_GWPT}
		\psi(x,0) =  2^{d/4}(\pi\varepsilon)^{-d/4}\exp\Big\{ (\iu/\varepsilon)\Big[ \big(x-q_0\big)^{T}\alpha_0\big(x-q_0\big)+p_0\big(x-q_0\big)+\gamma_0	 \Big] \Big\}.
	\end{equation}
	where $ q_0,p_0 $ are vectors in $ \mathbb{R}^{d\times 1} $, $\gamma_0 $ is a complex scalar and $\alpha_0$ is a complex-valued $ d\times d $ matrix whose imaginary part is positive-definite. Here $ \gamma_0 $ is chosen to satisfy $ \Vert \psi(\cdot,0) \Vert_{\mathbb{L}^2_x} = 1 $. Such normalization implies that 
	\begin{equation}\label{eqn:ini_val}
		\exp[ -\gamma_{0,I}/\varepsilon] = \{\det[\alpha_{0,I}]\}^{1/4}.
	\end{equation}
	where $ \alpha_{I} $ is the imaginary part of the complex matrix $ \alpha $. And it is clear from the above arguments that $ \alpha_{0,I} $ does not depend on $ \varepsilon $, while $ \gamma_{0,I} $ depends on $ \varepsilon $.
	This Cauchy problem has three typical scales: the $ \mathcal{O}(\varepsilon) $ oscillations in the phase, the $ \mathcal{O}(\sqrt{\varepsilon}) $ beam width, and its dynamics over a $ \mathcal{O}(1) $ space-and-time domain. The external scalar potential varies in a $ \mathcal{O}(1) $ scale as well. Many asymptotic approaches have been extensively studied, while the introduction of the asymptotic approximation error was inevitable. In \cite{Russo2013}, the authors proposed the Gaussian wave-packet transform (GWPT) to separate the three scales in the time evolution of highly oscillatory wave function, and the resulting system is equivalent to the original semi-classical Schr\"odinger equation.\par 
	We now investigate this method in detail. Suppose the solution is in the form of  
	\begin{eqnarray}\label{eqn:antzas}
		\psi(x,t)=W(\xi ,t)\exp \left[ (\iu/\varepsilon)\left( \xi^{T}\alpha_{R}\xi+p^{T}\xi+\gamma \right) \right],
	\end{eqnarray}
	where $ \xi=x-q\in \mathbb{R}^{d\times 1} \mbox{ and } \alpha_{R} \in \mathbb{R}^{d\times d} $ is the real part of a complex $d\times d$ matrix $ \alpha $. Furthermore suppose $ q,p,\gamma,\alpha $ are functions of time whose dynamics are governed by the following ordinary differential equations
	\begin{gather}\label{eqn:ODE}
		\begin{aligned}
			\dot{q}&=p,\quad\dot{p}=-(\nabla V(q))^{T},\quad  \\
			\dot{\gamma}&=\frac{1}{2}p^{T}p-V(q)+\iu\varepsilon\operatorname{tr}(\alpha_{R}),\\
			\dot{\alpha}&= \frac{d(\alpha_{R}+\iu\alpha_{I})}{dt} = -2\alpha\alpha -\frac{1}{2}\nabla\nabla V(q).\\
		\end{aligned}
	\end{gather}
	The initial values for $ q,p,\gamma,\alpha $ can be determined from the initial value \eqref{eqn:ini_GWPT}, and are, respectively, $ q_0,p_0,\gamma_0,\alpha_0 $. Here, $ q $ and $ p $ are interpreted as the position and momentum center of the Gaussian wave packet and they follow the Hamiltonian's equations of classical dynamics. Then the equation for $ W(\xi,t) $ becomes
	\begin{align}
		\partial_tW&=\frac{\iu\varepsilon}{2}\Delta_{\xi}W-2\xi^{T}\alpha_{R} (\nabla_{\xi}W)^{T}-\frac{\iu}{\varepsilon}F(\xi)W,
	\end{align}
	where 
	\begin{align*}
		F &=V_{2}(\xi;q)+2\xi^{T}\alpha_{I}\alpha_{I}\xi,\\
		V_2(\xi;q) &= V(q+\xi)-V(q)-\nabla V(q)\xi - \frac{1}{2}\xi^{T}\nabla\nabla V(q)\xi.
	\end{align*}
	Notice that $V_2$ is the difference between the potential and its second order Taylor expansion around $q$.
	The {\em ansatz\/} \eqref{eqn:antzas} and the ODE system \eqref{eqn:ODE} separate the $ \mathcal{O}(\varepsilon) $ oscillations and ensure that the $ W $ function is concentrated around $\xi=0$, i.e. where the position variable $x$ is close to the wave packet center $q$. Finally, we take the change of coordinates  
	\begin{equation}\label{eqn:change_coor}
		\eta=B\xi/\sqrt{\varepsilon},
	\end{equation} 
	where $ B $ is a function of time that satisfies the following ODE
	\begin{align}
		\dot{B} &= -2B\alpha_{R}.\\
		B_0 &= \sqrt{\alpha_{0,I}}.
	\end{align} 
	It can be shown that $ \alpha_{I} = B^{T}B $ for all $ t>0 $, thus the matrix $ B $ is invertible at any time. By the change of coordinates \eqref{eqn:change_coor}, the wave-packets has width $ \mathcal{O}(1) $ on $ \eta $ space, thus resolving the $ \mathcal{O}(\sqrt{\varepsilon}) $ scale. Now we obtain the equation for $ w(\eta,t)=W(\xi(\eta,t),t) $ and its initial value
	\begin{align}\label{eqn:wEqn}
		\iu\partial_t w&=-\frac{1}{2}\operatorname{tr}(B^{T}\nabla_{\eta}\nabla_{\eta}wB)+2\eta^{T}BB^{T}\eta w+\sqrt{\varepsilon} U(\eta;q)w\triangleq \hat{H}(t)w,\\
		w(\eta,0) &= \frac{2^{d/4}}{(\pi\varepsilon)^{d/4}} \exp ( -\eta^{T}\eta )\label{eqn:w_ini},
	\end{align}
	where $ U(\eta;q) = \varepsilon^{-3/2} V_2(\sqrt{\varepsilon}B^{-1}\eta;q) \sim \mathcal{O}(1)$ is given by
	\begin{equation}\label{eqn:potential}
		U(\eta;q) = \frac{1}{\varepsilon^{3/2}}[V(q + \sqrt{\varepsilon}B^{-1}\eta) - V(q) - \sqrt{\varepsilon} \nabla V(q)(B^{-1}\eta) - \frac{\varepsilon}{2}(B^{-1}\eta)^{T}\nabla\nabla V(q)B^{-1}\eta].
	\end{equation} 
	This Schr\"odinger-typed equation \eqref{eqn:wEqn} and the above ordinary differential equations \eqref{eqn:ODE} are not oscillatory, making them easier to solve than the original Cauchy problem \eqref{eqn:SEqn}-\eqref{eqn:Sini_val}. Another thing to notice here is that the $ \mathbb{L}^2_{\eta} $ norm of the initial value \eqref{eqn:w_ini} is $  \varepsilon^{-d/4} $. Besides, this transform can be extended to deal with more general initial conditions, see \cite{Russo2013,Russo2014} for more details. \par 
	To propose a numerical method based on the GWPT, we have to solve a set of ordinary differential equations \eqref{eqn:ODE} and the $ w $ equation \eqref{eqn:wEqn} efficiently. In \cite{Russo2013}, the authors proposed a numerical method based on the GWPT. They implement a Fourier spectral method to solve the $ w $ equation \eqref{eqn:wEqn} on some artificial computational domain $ [-\sigma/2,\sigma/2] $, which will bring in numerical errors that are difficult to quantify.\par  
	In the rest of the article, the parameters $ q,p,\gamma,\alpha_{R},\alpha_{I}, B $ will be called \emph{GWPT parameters}. And we will define the quadratic part $ \hat{H}_{q}(t) $ of the Hamiltonian operator $ \hat{H}(t) $ as
	\begin{equation}\label{eqn:Hamil_qua}
		\hat{H}_{q}(t)w \triangleq -\frac{1}{2}\operatorname{tr}(B^{T}\nabla_{\eta}\nabla_{\eta}wB)+2\eta^{T}BB^{T}\eta w.
	\end{equation}
	It is easy to see that both operators $ \hat{H}_{q}(t),\hat{H}(t) $ are self-adjoint.
	
	\subsection{Numerical Method Description}\label{subsec:formulation}
	
	In devising a numerical method based on the GWPT, the main issue concerns the construction of an efficient algorithm for the $ w $ equation \eqref{eqn:wEqn}. Notice first that the residue term $ U/\varepsilon $ given by 
	\eqref{eqn:potential} is order $ \mathcal{O}(\varepsilon^{1/2}) $, which can be treated as a perturbation. Therefore, one may expect to use superpositions of eigenfunctions of $ \hat{H}_{q}(t) $ to represent the solution, and the perturbation term determines the evolution of the expansion coefficients, which can be dealt with in a variational way. In \cite{Hagedorn1980}, Hagedorn proposed a constructive method to identify the eigenfunctions of time-dependent quadratic quantum Hamiltonians, which can be viewed as a generalization of the quantum harmonic oscillator theory. The eigenfunctions can referred to as the semi-classical wave-packets, or Hagedorn's wave packets. We will first review some basic properties of the semi-classical wave-packets. 
	\subsubsection{Review of Hagedorn's Wave-packets (HWP)}
	To construct such semi-classical wave-packets recursively with fixed real vectors $ q_h, p_h\in \mathbb{R}^{d\times 1}$ and invertible matrices $ Q_h,P_h\in \mathbb{C}^{d\times d} $ that satisfy
	\begin{align} \label{eqn:symp_relation1}
		Q_h^{T}P_h - P^{T}_hQ_h &= 0,\\
		Q_h^{\ast}P_h - P^{\ast}_hQ_h &= 2\iu E_{d} \triangleq 2\iu \times \operatorname{diag}(1,1,1\cdots,1) .\label{eqn:symp_relation2} 
	\end{align}
	Here $ Q_{h}^{T},P_{h}^{T} $ denote the transpose of $ Q_h,P_h $ and $ Q_{h}^{\ast},P_{h}^{\ast} $ denote the conjugate transpose of $ Q_h,P_h $. Vector $ q_h,p_h $ can be interpreted as the center of the wave-packet on position and momentum space, while the matrices $ Q_h,P_h $ are related to the width of the wave-packets. We add subscript \lq\lq$ h $\rq\rq~to denote the HWP parameters, which are to be distinguished from the GWPT parameters. Based on them, Hagedorn constructed a set of orthonormal basis functions, also known as the semi-classical wave packets $ \{ \varphi_{k}^{\delta} \}_{k\in\mathbb{N}^d} $ in $ \mathbb{L}_{x}^{2}(\mathbb{R}^d) $ \citep[see][]{Hagedorn1980,Hagedorn1998} based on a generic quadratic Hamiltonian operator $ \hat{H} $. Here $\delta$ denotes the semi-classical scale of the HWP, while $\sqrt{\delta}$ is related to the typical width of the HWP. 
	In particular, in 1D, those basis functions coincide with the eigenfunction of the Hamiltonian operator. It is worth noting that $ H_q $ as in \eqref{eqn:Hamil_qua} is indeed one example of such Hamiltonians. The construction of the basis functions is carried out by the raising and lowering operators, which are given by $ \mathcal{A}^{\dagger}= ( \mathcal{A}_{j}^{\dagger} )_{j=1}^{d} $ and $ \mathcal{A} = (\mathcal{A}_{j} )_{j=1}^{d} $
	\begin{align}\label{eqn:lowering_op}
		\mathcal{A}[q_h,p_h,Q_h,P_h] &\triangleq \frac{1}{\sqrt{2\delta}}\left[ \iu Q^{T}_h(\hat{p}-p_h)-\iu P_h^{T}(x-q_h) \right],\\
		\mathcal{A}^{\dagger}[q_h,p_h,Q_h,P_h]&\triangleq \frac{1}{\sqrt{2\delta}}\left[ -\iu Q^{\ast}_h(\hat{p}-p_h)+\iu P_h^{\ast}(x-q_h) \right],
		\label{eqn:raising}
	\end{align} 
	where $ \hat{p}\triangleq -\iu\delta \nabla_x $ is the momentum operator. 
	The first wave-packet is given by
	\begin{align}\label{eqn:first_mode}
		\varphi^{\delta}_{0}[q_h,p_h,Q_h,P_h](x)
		=(\pi\delta)^{-d/4}(\operatorname{det}Q_h)^{-1/2}\exp\left( \frac{\iu}{2\delta}(x-q_h)^{T}P_hQ_h^{-1}(x-q_h) + \frac{\iu}{\delta}p_h^{T}(x-q_h) \right),
	\end{align}
	while other wave-packets, labeled by multi-index $ k=(k_1,k_2,\cdots,k_d) $ with non-negative integer entries are given by
	\begin{align}
		\varphi_{k+\langle j \rangle}^{\delta} &= \frac{1}{\sqrt{k_j+1}}\mathcal{A}^{\dagger}_{j}\varphi^{\delta}_{k},\\
		\varphi_{k-\langle j \rangle}^{\delta} &= \frac{1}{\sqrt{k_j}}\mathcal{A}_{j}\varphi^{\delta}_{k}. 
	\end{align}
	Here $ \langle j \rangle $ denotes the $ j$-th unit vector in $ d $-dimension and $j$ takes its value in $1,2,\cdots,d$. $ \varphi_{k-\langle j \rangle} $ is set to be zero if one of the coordinates in multi-index $ k-\langle j \rangle $ is negative. And it can be proved that $ \{\varphi_{k}^{\delta}\}_{k\in\mathbb{N}^d} $ constitute an orthonormal basis of $ \mathbb{L}^2_x(\mathbb{R}^{d}) $, which can be also related to the quantum Hamiltonian operator defined in \eqref{eqn:Hamil_qua}. From the above two operators, we can give the following three term recurrence relations (see, for example, formula (3.28) in \cite{Hagedorn1998})
	\begin{align}\label{eqn:Three_rec}
		x-q_h = \sqrt{\frac{\delta}{2}}\left(  Q_h\mathcal{A}^{\dagger} + \overline{Q}_h \mathcal{A}\right).
	\end{align}
	This formula can be more explicitly written as 
	(see \cite{Lasser2020} for detail)
	\begin{align}\label{eqn:Three_rec2}
		Q_h \left(\sqrt{k_j+1}\varphi_{k+\langle j \rangle}^{\delta}(x)\right)_{j=1}^{d} = \sqrt{\frac{2}{\delta}}(x-q_h)\varphi_{k}^{\delta}(x) - \overline{Q}_h\left(\sqrt{k_j}\varphi_{k-\langle j \rangle}^{\delta}(x)\right)_{j=1}^{d}.
	\end{align}
	Here $ \left(\sqrt{k_j+1}\varphi_{k+\langle j \rangle}^{\delta}(x)\right)_{j=1}^{d} $ and $ \sqrt{\frac{2}{\delta}}(x-q_h)\varphi_{k}^{\delta}(x) $ are regarded as column vectors whose $ j $-th component are $ \sqrt{k_j+1}\varphi_{k+\langle j \rangle}^{\delta}(x) $ and $ \sqrt{\frac{2}{\delta}}(x_j-q_{h,j})\varphi_{k}^{\delta}(x) $ respectively. $ \varphi_{k-\langle j \rangle} $ is set to be zero if one of the coordinates in multi-index $ k-\langle j \rangle $ is negative. $ \overline{Q}_h $ denotes the element-wise conjugation of matrix $ Q_h $. The recurrence relation provides a practical method to compute the basis $ \varphi_{k}^{\delta} $ for any multi-index $ k $ with fixed $ \delta $. It is observed by the author \citep[see][]{Hagedorn1998} that Hagedorn wave-packets in 1d are rescaled Hermite functions. And the first step to prove the convergence of fully discrete HWP methods starts from proving the spectral convergence of the fully discrete Gauss-Hermite spectral method, see \cite{Faou2008} for example. \par 
	In \cite{Faou2009}, the authors use finite superposition of $ \{ \varphi^{\varepsilon}_{k} \}_{k\in\mathbb{N}^d} $ to represent the solution of the original equation \eqref{eqn:SEqn} directly, while there is no rigorous analysis in the article. In \cite{Zhou2014}, it is shown that the error is order $ \mathcal{O}(\varepsilon^{k/2}) $, while $ k $ is related to the number of wave-packets. In the following sections, we will propose the GWPT+HWP method based on the aforementioned methods. 
	
	\subsubsection{Formulation of the GWPT+HWP Method--quadratic case}
	To explain how the method works, let's start from a simple case when $ V(x) $ is a second degree polynomial, so that the residue potential $ U $ as in \eqref{eqn:potential} vanishes. Then the Hamiltonian operator $ \hat{H}(t) $ in \eqref{eqn:wEqn} can be written in the form of
	\begin{align}
		\hat{H}(t) = \hat{H}_q(t) &= \frac{1}{2}\begin{pmatrix}
			-\iu\nabla_x\\x
		\end{pmatrix}^{T} \begin{pmatrix}
			BB^{T}     & 0\\
			0 & 4BB^{T}
		\end{pmatrix}\begin{pmatrix}
			-\iu\nabla_x\\x
		\end{pmatrix}.
	\end{align}
	It is proved in \cite{Hagedorn1998} that if we let $ q_h,p_h,Q_h,P_h $ be time-dependent parameters and introduce another time-dependent parameter $ S_h\in\mathbb{R} $ such that 
	\begin{gather}\label{eqn:HWP_eqn}
		\begin{aligned}
			\dot{q}_h &= BB^{T}p_h, \quad\dot{p}_h = -4BB^{T}q_h,\quad\dot{S}_h = \frac{1}{2} p_h^{T}BB^{T}p_h - 2 q_h^{T}BB^{T}q_h,\\
			\dot{Q}_h &= BB^{T}P_h,\quad \dot{P}_h = -4BB^{T}Q_h,
		\end{aligned}
	\end{gather}
	and suppose further that the initial value can be represented by a convergent series
	\begin{equation}
		w(\eta,0) = \eu^{\iu S_h(0)}\sum_{k\in\mathbb{N}^{d}}c_k(0)\varphi^{1}_k[q_h(0),p_h(0),Q_h(0),P_h(0)](\eta),\quad \{c_{k}\}_{k\in\mathbb{N}^d}\subset\mathbb{C},
	\end{equation}
	then, the solution to equation \eqref{eqn:wEqn} is given by
	\begin{equation}\label{eqn:sol_Harmonic}
		w(\eta,t) = \eu^{\iu S_h(t)}\sum_{k\in\mathbb{N}^{d}}c_k(0)\varphi^{1}_k[q_h(t),p_h(t),Q_h(t),P_h(t)](\eta).
	\end{equation} 
	It can be proved that $ Q_h(t),P_h(t) $ are invertible and relation \eqref{eqn:symp_relation1}-\eqref{eqn:symp_relation2} holds for all $ t>0 $, which guarantees the existence of the first wave-packet. In the case when the initial value is in the form of \eqref{eqn:w_ini}, we can derive the initial values for \eqref{eqn:HWP_eqn}
	\begin{gather}\label{eqn:HWP_ini}
		\begin{aligned}
			q_h(0),p_h(0) &= 0,\quad Q_h(0) = \frac{1}{\sqrt{2}}E_d,\quad P_h(0) = \sqrt{2}\iu E_d,\quad S_h(0) =0.\\
			c_{0}(0)&=\varepsilon^{-d/4},\quad c_{k}(0)=0,\;k\neq 0.
		\end{aligned}
	\end{gather}
	which implies that the sum in \eqref{eqn:sol_Harmonic} is finite. With a further inspection, we see that only $ Q_h,P_h $ depend on time, while $ q_h,p_h,S_h,c_k $ can be shown to be constant in time.\par 
	Therefore, when the external potential $ V $ is quadratic, the full semi-classical quantum dynamics are equivalently recast as the dynamics of the GWPT parameters and the HWP parameters, while significantly less computational cost is needed for simulating the latter systems. 
	\subsubsection{Formulation of the GWPT+HWP method--General Case} 
	To approximate the $ w $ equation \eqref{eqn:wEqn} with general potential function, the residue term $ U $ in \eqref{eqn:wEqn} won't vanish. But recalling that semi-classical wave-packets form a basis of $ \mathbb{L}_x^2(\mathbb{R}^d) $, we can represent the exact solution $ w(\eta,t) $ by an infinite superposition of wave-packets with time-dependent expansion coefficients
	\begin{align}\label{eqn:summation}
		w(\eta,t) = \eu^{\iu S_h(t)}\sum_{k\in\mathbb{N}^{d}}c_k(t)\varphi^{1}_k[q_h(t),p_h(t),Q_h(t),P_h(t)](\eta). 
	\end{align}
	Here $ q_h,p_h,Q_h, P_h,S_h$ are determined by \eqref{eqn:HWP_eqn} and $ \{c_k\}_{k\in \mathbb{N}^d} $ are time-dependent coefficients. To determine the dynamics for the coefficients $ \{ c_k \}_{k\in\mathbb{N}^d} $ from \eqref{eqn:wEqn}, we solve an infinite system of ordinary differential equations 
	\begin{align}\label{eqn:Full_Galerkin}
			\iu\dot{c}_k = \sqrt{\varepsilon} \sum_{l\in\mathbb{N}^{d}}f_{kl}c_l,\quad k\in \mathbb{N}^{d}.
		\end{align}
		where $ \{f_{kl}\}_{k,l\in\mathbb{N}^{d}} $ is an infinite dimensional Hermitian matrix whose elements are determined by
		\begin{align}\label{eqn:Fmatrix}
			f_{kl}(t) = \int_{\mathbb{R}^{d}}\overline{\varphi}^{1}_k(\eta,t)U(\eta,t)\varphi^{1}_l(\eta,t)d\eta,\quad \forall \,k,l\in \mathbb{N}^{d}.
	\end{align}
	It is clear that solving 
	\eqref{eqn:Full_Galerkin}-\eqref{eqn:Fmatrix} and \eqref{eqn:HWP_eqn} is equivalent to solving the Cauchy problem for \eqref{eqn:wEqn}. \par 
	To construct a practical algorithm, we can truncate the series representation in \eqref{eqn:summation}, obtaining the approximation
	
	\begin{align}
		\quad \tilde{w}(\eta,t) = \eu^{\iu S_h}\sum_{k\in\mathcal{K}}c_k(t)\varphi^{1}_k[q_h,p_h,Q_h,P_h](\eta),\quad \#(\mathcal{K})<+\infty,
	\end{align}
	and the dynamics of the coefficients are given by 
	\begin{align}
		\iu\dot{c}_k = \sqrt{\varepsilon} \sum_{l\in\mathcal{K}}f_{kl}c_l,\quad k\in \mathcal{K}.\label{eqn:Galerkin_dynamics}
	\end{align}
	where $ F \triangleq\{f_{kl}\}_{k,l\in\mathcal{K}} $ is determined in \eqref{eqn:Fmatrix} and can be approximated by certain quadrature rule. 
	\begin{remark}\label{rmk:quad_free}
		We remark here that the matrix $ F = \{ f_{kl} \}_{k,l\in\mathcal{K}} $ can be evaluated by a recurrence-free method if the residue $U(\eta,t)$ can be approximated by polynomial functions. Equation \eqref{eqn:Three_rec} tells us that the polynomials can be represented by a linear combination and composition of raising and lowering operators, thus forming the $ F $ matrix can be achieved with far less computations than direct computation of the integral in Eq.~\eqref{eqn:Fmatrix}. But we will leave this topic to our next project.
	\end{remark}
	\begin{remark}\label{rmk:variables}
		Compared with the quadratic case, we only introduce the dynamics of $\{ c_k \}_{k\in \mathbb{N}^{d}}$, so $ q_h,p_h, S_h $ are still constants. Specifically, $ q_h = p_h = 0$.  
	\end{remark}
	To integrate the ODE \eqref{eqn:Galerkin_dynamics} with time step $ \Delta t_c $, one needs GWPT parameters $ q,p,\gamma,\alpha_{R},\alpha_{I}, B $, parameters $ Q_{h},P_{h} $ and evaluations of recurrence relations \eqref{eqn:Three_rec} on some finer time grid with step $\Delta t_{gt}$ 
	with higher order accuracy. Therefore, to update the expansion coefficients $ \{c_k\}_{k\in\mathcal{K}} $ with time step $ \Delta t_c $, the three term recurrence relations \eqref{eqn:Three_rec2} should be estimated on time step of $ \mathcal{O}(\Delta t_c) $. 
	
	From the arguments above, we can give our algorithm:
	\begin{algorithm}
		\caption{GWPT+HWP method to advance by $\Delta t_c$}\label{alg:new_method}
		\begin{algorithmic}[1]
			\State Update the GWPT parameter $ q,p,\gamma,\alpha_{R},\alpha_{I}, B $ and parameters $ Q_{h},P_{h} $ by fourth order RK4 with time step $\Delta t_{gt}$, from current ti    me $t$ to time $t+\Delta t_c$. 
			\State Apply the three-term recurrence relation \eqref{eqn:Three_rec} to calculate $ \{ \varphi^{1}_{k}(\cdot,t) \}_{k\in \mathcal{K}} $ on quadrature points on time step of order $ \mathcal{O}(\Delta t_c) $. 
			\State Update the coefficient vector $ c = \{c_{k}\}_{k\in \mathcal{K}} $ by solving the ODE:
			\begin{align*}
				\iu\dot{c} = \sqrt{\varepsilon}F(t)c,
			\end{align*}
			by RK4 with time step $\Delta t_{c}$. $ F $ is an Hermitian matrix whose elements are given by \eqref{eqn:Fmatrix}
		\end{algorithmic}
	\end{algorithm}\par
	To conclude, this method uses superposition of semi-classical wave-packets to represent the solution, which allows us to deal with the quadratic and perturbation part respectively. This method can also be carried out with sparse grid techniques in several dimensions, which presents a possibility to solve \eqref{eqn:wEqn} in moderately high dimension. What is more, the method only uses wave-packets with fixed width over the computational domain when implemented with GWPT. This fact facilitates the numerical analysis. We will show that this method can resolve the three typical scales, as mentioned in Section \ref{subsec:review_GWPT}.
	\begin{remark}
		As shown later, $ \Delta t_{c} $ can be chosen independent of $ \varepsilon $, while $ \Delta t_{gt} $ needs to be small. And the cost of solving the ODE system is negligible compared with the cost of solving the Galerkin equation \eqref{eqn:Galerkin_dynamics}. So we can choose $\Delta t_{gt}$ such that the error of  we conclude that this method can achieve high accuracy while saving computational cost, when $ \varepsilon\ll 1 $. 
	\end{remark}
	
	\subsubsection{Numerical Complexity of GWPT+HWP}
	
	Now we analyze the numerical complexity of \Cref{alg:new_method}, using quadrature rules to evaluate the integral in \eqref{eqn:Fmatrix}. Here we suppose the total number of quadrature points is $ N_{Q} $. First, we present two possible choices of wave-packets in multiple dimension, both of which can be characterized with a parameter $ n $:
	\begin{align}
		\label{eqn:full_grid}
		\mathcal{K}^{n}_{\infty} &= \{ \varphi_{k}: |k|_{\infty}\le n, k\in \mathbb{N}^{d} \}, \\
		\label{eqn:tri_grid}
		\mathcal{K}^{n}_{1} &= \{ \varphi_{k}: |k|_{1}\le n, k\in \mathbb{N}^{d} \}.
	\end{align} 
	We will show later in the article that $ \mathcal{K}_{1}^{n} $ allows higher accuracy for the same computational cost, see Table \ref{table:FT}. So the number of wave-packets is $ C^{n}_{d+n} $ in the case of $ \mathcal{K}^{n}_{1} $, where $ C^{m}_{n} $ denotes the binomial coefficient representing picking $ m $ unordered outcomes from $ n $ possibilities. We will also show how this choice of wave-packets affects the complexity. \par
	To analyse the numerical complexity, we first recover what sort of operations are involved in each step in \Cref{alg:new_method}. In the first step, the GWPT parameters are calculated. Matrix-vector multiplication takes most of the effort in solving the ODE system \eqref{eqn:HWP_eqn} and \eqref{eqn:ODE}. So the overall cost will be $ \mathcal{O}(N_{gt}d^3) $, where $ N_{gt} = t_f/\Delta t_{gt} $ is the total time steps in solving the parameters $ q,p,\gamma,\alpha_{R},\alpha_{I}, B $ and $ Q_{h},P_{h} $. \par 
	In the second step, we shall assume that the inverse operation when evaluating wave-packets on quadrature points:
	$$ \left( \sqrt{k_j+1}\varphi^{\varepsilon}_{k+\langle j \rangle} \right)_{j=1}^{d} = Q_h^{-1}\left( \sqrt{\frac{2}{\varepsilon}}(x-q_h)\varphi^{\varepsilon}_{k} - \overline{Q}\big( \sqrt{k_j}\varphi^{\varepsilon}_{k-\langle j \rangle} \big)_{j=1}^{d}\right). $$
	will cost $ \mathcal{O}(d^2N_{Q}) $. About $ C_{d+n}^{n} $ three-term recurrence relations are needed to recover the full wave-packets in $ \mathcal{K}^{n}_{1} $, so the total calculation will cost $ \mathcal{O}\big( C_{d+n}^{n}d^2N_{Q} \big) $ operations.\par 
	In the last step in each time loop, we need to calculate the Galerkin matrix $ \{f_{kl}\}_{k,l\in\mathcal{K}} $ given by \eqref{eqn:Fmatrix}. To this purpose, one has to calculate the potential function \eqref{eqn:potential} and the wave-packets on a total number of $ N_Q $ quadrature points. $ \mathcal{O}(d^2N_{Q}) $ operations are needed to evaluate the potential on each quadrature points. Then from the same procedure to estimate the complexity of reconstructing wave-packets on quadrature points, we have to use $ \mathcal{O}\big[ C_{d+n}^{n}N_{Q} + (C_{d+n}^{n})^2N_{Q} \big] $ operations in the last step. \par 
	
	Combining the above arguments,  the total computational cost is: 
	\begin{equation}\label{eqn:complexity}
		\mathcal{O}\Big\{ N_{c}C_{d+n}^{n}d^2N_{Q} + N_c(C_{d+n}^{n})^2N_{Q} +N_{gt}d^3\Big\},
	\end{equation}
	where $ N_c = t_f / \Delta t_c $ is the total time step in solving the coefficients $ \{ c_{k} \}_{k\in \mathcal{K}} $. \par 
	In the above analysis, the number of $ C_{d+n}^{n} $ is reduced largely compared with $ n^{d} $ if $ d\gg n $ and $ N_{Q} $ is at the same scale as $ C_{d+n}^{n} $. This fact follows from the asymptotes for binomial number by Stirling's formula 
	\begin{align}
		C^{n}_{d+n} \sim \bigg(1+\frac{n}{d}\bigg)^{d+\frac{1}{2}}\frac{(n+d)^{n}}{n!\eu^n} \sim \frac{(n+d)^{n}}{n!} ,
	\end{align}
	which grows polynomially with respect to $ d $. Now supposing $ N_Q $ grows polynomially in $ d $, then from \eqref{eqn:complexity}, the overall complexity grows polynomially with respect to $ d $. We will show later that this method will converge faster with respect to the number of wave-packets, and thus requires far less number of wave-packets compared with the Fourier spectral method. 
	\begin{remark}
		According to \Cref{rmk:quad_free}, a quadrature-free method might be developed to evaluate the $ F $ matrix with improved efficiency. In that scenario, we expect that the complexity can be further reduced. Whereas, how the complexity depends on the dimensionality is still yet to be explored, and we shall work on that in the future as well.
	\end{remark} 
	
	\section{Convergence Analysis of the GWPT+HWP Method}\label{sec:Numerical_Analysis}
	Since there is no rigorous convergence analysis for GWPT-based numerical methods prior to this work, we shall estimate the error of the GWPT+HWP method in two steps. First, we establish an error estimate framework in Section \ref{subsec:GWPT_estimate} for generic GWPT-based methods in assembling the errors in approximating the GWPT parameters \eqref{eqn:ODE} and the $w$ equation \eqref{eqn:wEqn} respectively. Considering the convergence analysis for the numerical approximation of the GWPT parameters is rather standard, the error estimate boils down to quantifying the accuracy of approximating the $w$ equation by the Hagedorn's wave-packets when $\delta=1$, which we will show to be spectrally accurate with respect to the number of wave-packets in one dimension in Theorem \ref{thm:approxi}.  \par 
	
	\subsection{Error Estimate of the GWPT}\label{subsec:GWPT_estimate}
	We first study how the errors in solving the $ w $ equation \eqref{eqn:wEqn}-\eqref{eqn:w_ini} and the GWPT parameters \eqref{eqn:ODE} affect the error in solving the semi-classical Schr\"odinger equation \eqref{eqn:SEqn}-\eqref{eqn:Sini_val}. \par 
	We begin by showing the following identity. 
	\begin{lemma}\label{lem:clean}
		If the initial wave function is normalized, then for any $ t>0 $, the following identity holds $$ (\det(\alpha_{I}))^{1/4}= \exp(-\gamma_{I}/\varepsilon).$$  
	\end{lemma}
	\begin{proof}[Proof of Lemma \ref{lem:clean}]
		One only has to check that $ (\det(\alpha_{I}))^{1/4} $ and $ \exp(-\gamma_{I}/\varepsilon) $ share the same equation and initial value. Their initial values coincide by \eqref{eqn:ini_val}. \par 
		And by chain rule and \eqref{eqn:ODE}, one can derive the equation for $ \exp(-\gamma_{I}/\varepsilon) $.
		$$ \frac{d}{dt}\exp(-\gamma_{I}/\varepsilon) = -\frac{1}{\varepsilon}\exp(-\gamma_{I}/\varepsilon)(\varepsilon \operatorname{tr}(\alpha_R)) = -\operatorname{tr}(\alpha_{R}) \exp(-\gamma_{I}/\varepsilon). $$
		Then another direct calculation leads to
		\begin{align*}
			\frac{d}{dt}(\det(\alpha_I))^{1/4} = \frac{1}{4\det(\alpha_{I})^{3/4}}\frac{d}{dt}(\det(\alpha_I)) = -\operatorname{tr}(\alpha_R)(\det(\alpha_I))^{1/4}.
		\end{align*}\par 
		The uniqueness of the solution to an ordinary differential equation gives the desired result.
	\end{proof}
	
	We denote the exact solution of \eqref{eqn:wEqn} and  \eqref{eqn:ODE} by $ w $ and $ q,p,\gamma,\alpha,B $ respectively. Furthermore, let $ \tilde{w} $ and $ \tilde{q},\tilde{p},\tilde{\gamma},\tilde{\alpha},\tilde{B} $ denote the approximate solutions with error bounds given by
	\begin{align}
		\max_{t\in[0,T]}&\Vert w-\tilde{w} \Vert_{\mathbb{L}^2_{\eta}} \le E_1\varepsilon^{-d/4},\label{eqn:err_PDE}\\
		\max_{t\in [0,T]}\max\big(\Vert q-\tilde{q} \Vert_{\infty},&\Vert p - \tilde{p} \Vert_{\infty},\vert \gamma-\tilde{\gamma} \vert,\Vert \alpha-\tilde{\alpha} \Vert_{\infty},\Vert B-\tilde{B}\Vert_{\infty} \big)\le E_2.\label{eqn:err_ODE}
	\end{align}
	For a complex $d\times d$ matrix $A$ the standard natural norm is adopted:
	$$
	\|A\|_\infty = \max_{\|v\|=1}\|A v\|_\infty, \quad v\in \mathbb{C}^d.
	$$
	Here, $ E_2 $ denotes the error in solving the GWPT parameters, and  $ E_1\varepsilon^{-d/4} $ is the error in solving the $ w $ function, whose $ \mathbb{L}^2 $ norm is $\mathcal{O} \left( \varepsilon^{-d/4} \right)$ manifesting the magnitude of the initial condition for $w$ as in \eqref{eqn:w_ini}. \par  
	To investigate where the error comes from when rebuilding the wave function by approximate solutions $ \tilde{w} $ and $ \tilde{q},\tilde{p},\tilde{\gamma},\tilde{\alpha},\tilde{B} $, let's first recall some facts about GWPT. Notice the wave functions $ \psi, \tilde{\psi} $ are reconstructed by 
	\begin{align}
		\psi(x,t)&=w(\eta(x),t)\exp\{ (\iu/\varepsilon)\theta(x,t) \},\\
		\tilde{\psi}(x,t)&=\tilde{w}(\tilde{\eta}(x),t)\exp\{ (\iu/\varepsilon)\tilde{\theta}(x,t) \}.
	\end{align}
	where 
	\begin{align}\label{eqn:change_coor1}
		\tilde{\eta}(x)&=\varepsilon^{-1/2}\tilde{B}(x-\tilde{q}) ,\\
		\eta(x) &= \varepsilon^{-1/2}B(x-q),\label{eqn:change_coor2}\\
		\theta ( x,t  )&=(x-q(t))^{T}\alpha_{R}(x-q(t)) +p^{T}(t)(x-q(t))+\gamma_{R}(t)+\iu\gamma_{I}(t),\\
		&=  \theta_{R} + \iu\theta_{I}, \notag\\
		\tilde{\theta} ( x,t  )
		&=(x-\tilde{q}(t))^{T}\tilde{\alpha}_{R}(x-\tilde{q}(t) +\tilde{p}^{T}(t)(x-\tilde{q}(t))+\tilde{\gamma}_{R}(t)+\iu\tilde{\gamma}_{I}(t),\\
		&= \tilde{\theta}_R + \iu\tilde{\theta}_{I}. \notag
	\end{align} 
	By triangle inequality
	\begin{gather}\label{eqn:err_GWPT}
		\begin{split}
			\Vert\psi-\tilde{\psi}\Vert_{\mathbb{L}_x^2} 
			\le{}&\Vert w(\eta(x))\exp\{(\iu/\varepsilon)\theta(x)\} -w(\eta(x))\exp\{(\iu/\varepsilon)\tilde{\theta}(x)\}\Vert_{\mathbb{L}_x^2}  \\
			&+ \Vert w(\eta(x))\exp\{ (\iu/\varepsilon)\tilde{\theta}(x) \}- w(\tilde{\eta}(x))\exp\{(\iu/\varepsilon)\tilde{\theta}(x)\} \Vert_{\mathbb{L}_x^2}\\&+\Vert w(\tilde{\eta}(x))\exp\{(\iu/\varepsilon)\tilde{\theta}(x)\}-\tilde{w}(\tilde{\eta}(x))\exp\{(\iu/\varepsilon)\tilde{\theta}(x)\}\Vert_{\mathbb{L}_x^2} .
		\end{split}
	\end{gather}
	Here $ t $ is omitted for simplicity. \par 
	To deal with the third part of RHS in \eqref{eqn:err_GWPT}, where the error in $ w $ equation is dominant, one first notices that, by definition,
	\begin{align*}
		\Vert w(\tilde{\eta}(x))\exp\{(\iu/\varepsilon)\tilde{\theta}\}-\tilde{w}(\tilde{\eta}(x))\exp\{(\iu/\varepsilon)\tilde{\theta}\}\Vert_{\mathbb{L}_x^2} = \exp(-\tilde{\gamma}_{I}/\varepsilon)\Vert w(\tilde{\eta}(x))-\tilde{w}(\tilde{\eta}(x))\Vert_{\mathbb{L}_x^2}.
	\end{align*}
	Then a change of coordinates leads to 
	\begin{align}\label{eqn:eta&tilde}
		\exp(-\tilde{\gamma}_{I}/\varepsilon)\Vert w(\tilde{\eta}(x))-\tilde{w}(\tilde{\eta}(x))\Vert_{\mathbb{L}^2_x}   =\sqrt[4]{\frac{\varepsilon^{d}}{\det(\tilde{B})^2}}\exp(-\tilde{\gamma}_{I}/\varepsilon)\sqrt{\int_{\mathbb{R}^{d}}|w-\tilde{w}|^2d\tilde{\eta}}.
	\end{align}
	Combining \eqref{eqn:change_coor1} and \eqref{eqn:change_coor2}, we have \begin{align}
		\tilde{\eta} = \tilde{B}B^{-1}\eta + \varepsilon^{-1/2}\tilde{B}(q-\tilde{q}).
	\end{align}
	And it follows that 
	\begin{align}
		\sqrt[4]{\frac{\varepsilon^{d}}{\det(B)^2}}\exp(-\tilde{\gamma}_{I}/\varepsilon)\sqrt{\int_{\mathbb{R}^{d}}|w-\tilde{w}|^2d\eta} \leq \frac{\exp(-\tilde{\gamma}_{I}/\varepsilon)}{[\det(\alpha_{I})]^{1/4}}{E_1}.
	\end{align}\par
	Then, we investigate the second part of the RHS in \eqref{eqn:err_GWPT}, where the error is mainly due to the coordinate transformation \eqref{eqn:change_coor1}. From \eqref{eqn:eta&tilde}, we have
	\begin{gather}
		\begin{aligned}
			\Vert w(\tilde{\eta}(x))&\exp\{(\iu/\varepsilon)\tilde{\theta}(x)\} -w(\eta(x))\exp\{(\iu/\varepsilon)\tilde{\theta}(x)\}\Vert_{\mathbb{L}_x^2}\\ &=\sqrt[4]{\frac{\varepsilon^{d}}{\det(\alpha_{I})}}\exp(-\tilde{\gamma}_{I}/\varepsilon) \Vert w(\eta) - w(\tilde{B}B^{-1}\eta + \varepsilon^{-1/2}\tilde{B}(q-\tilde{q})) \Vert_{\mathbb{L}^2_{\eta}}.
		\end{aligned}
	\end{gather}
	By the triangle inequality, we have
	\begin{align}
		C^{\varepsilon}(t)\Vert w(\eta) - w(\tilde{B}B^{-1}\eta&+ \varepsilon^{-1/2}\tilde{B}(q-\tilde{q})) \Vert_{\mathbb{L}^2_{\eta}} \le C^{\varepsilon}(t) \Big[\Vert w(\eta) - w(\eta + \varepsilon^{-1/2}\tilde{B}(q-\tilde{q})) \Vert_{\mathbb{L}^2_{\eta}} \notag\\ 
		& \quad+ \Vert w(\eta + \varepsilon^{-1/2}\tilde{B}(q-\tilde{q})) - w(\tilde{B}B^{-1}\eta +\varepsilon^{-1/2}\tilde{B}(q-\tilde{q})) \Vert_{\mathbb{L}^2_{\eta}}\Big]\label{eqn:second_part}.
	\end{align}
	Here $ C^{\varepsilon}(t) $ equals to $ \sqrt[4]{\frac{\varepsilon^{d}}{\det(\alpha_{I})}}\exp(-\tilde{\gamma}_{I}/\varepsilon) $ and is bounded above for all $ t\in [0,T] $. Provided $ w\in\mathbb{H}^{1}(\mathbb{R}^{d}) $ and 
	\begin{align}\label{eqn:err_H1norm}
		\Vert \nabla w \Vert_{\mathbb{L}^2_{\eta}} \le C_1(T,d)\varepsilon^{-d/4}, \quad \forall t\in[0,T],
	\end{align}
	we can estimate the first part of RHS in \eqref{eqn:second_part} as follows
	\begin{align}
		C^{\varepsilon}(t)\Vert w(\eta) - w(\eta + \varepsilon^{-1/2}\tilde{B}(q-\tilde{q})) \Vert_{\mathbb{L}^2_{\eta}} &\le C^{\varepsilon}(t)\Vert \nabla w(\eta) \Vert_{\mathbb{L}^2_{\eta}} \Vert \varepsilon^{-1/2}\tilde{B}(q-\tilde{q})\Vert_{2} \notag\\
		& \le C^{\varepsilon}(t)\Vert \nabla w(\eta) \Vert_{\mathbb{L}^2_{\eta}}  \varepsilon^{-1/2}E_2\Vert\tilde{B}\Vert_{2} \notag\\
		& \le \frac{C_2(T,d)\exp(-\tilde{\gamma}_{I}/\varepsilon)}{\sqrt[4]{\det(\alpha_{I})}} \sqrt{\varepsilon} E_2 ,\label{eqn:THM1_est1}
	\end{align} 
	as we are going to prove below. 
	The first inequality follows from Proposition 9.3 in \cite{Brezis2010}. 
	\begin{lemma}[\cite{Brezis2010}, Proposition 9.3, page 267]
		Suppose $ u\in \mathbb{H}_{x}^{1}(\mathbb{R}^{d}) $, then there exist a constant $ C_3 = \Vert \nabla u  \Vert_{\mathbb{L}^{2}_x} $ such that for all $ h\in\mathbb{R}^{d} $ we have
		$$  \Vert u(x + h) - u(x) \Vert_{\mathbb{L}^{2}_{x}}  \le C_3 \Vert h \Vert_{2}  $$
	\end{lemma}To estimate the second part of RHS in \eqref{eqn:second_part} we will first give a lemma
	\begin{lemma}\label{lem:estimate_moment}
		Suppose $ u\in \mathcal{S}(\mathbb{R}^d) $, we have 
		\begin{align}
			\Vert u(\eta + \varepsilon^{-1/2}\tilde{B}(q-\tilde{q})) - u(\tilde{B}B^{-1}\eta +\varepsilon^{-1/2}\tilde{B}(q-\tilde{q})) \Vert_{\mathbb{L}^2_{\eta}} \le C_4(T,d) \Vert \eta \nabla u \Vert_{\mathbb{L}^2_{\eta}}E_2.
		\end{align}
		Here $ \Vert \eta \nabla u \Vert_{\mathbb{L}^2_{\eta}} $ denotes the quantity
		\begin{align}
			\Vert \eta \nabla u \Vert_{\mathbb{L}^2_{\eta}} \triangleq \left(\int_{\mathbb{R}^{d}} |\nabla u|^2 |\eta|^2d\eta \right)^{1/2}.
		\end{align} 
	\end{lemma} 
	\begin{proof}[Proof of \Cref{lem:estimate_moment}]
		By Taylor's expansion with integral reminder, and a simple identity 
		\begin{align}
			\tilde{B}B^{-1} = E_d + (\tilde{B} - B)B^{-1},
		\end{align}
		we have
		\begin{align*}
			u(\tilde{B}B^{-1}\eta +\varepsilon^{-1/2}\tilde{B}(q-\tilde{q}))&-u(\eta + \varepsilon^{-1/2}\tilde{B}(q-\tilde{q}))  \\&=  \int_{0}^{1} \Big[\nabla u(\eta + \lambda (\tilde{B} - B)B^{-1}\eta) \cdot (\tilde{B} - B)B^{-1}\eta \Big]d\lambda.
		\end{align*}
		Then by Cauchy's inequality
		\begin{align*}
			\Vert u(\tilde{B}B^{-1}\eta +\varepsilon^{-1/2}\tilde{B}(q-\tilde{q}))&-u(\eta + \varepsilon^{-1/2}\tilde{B}(q-\tilde{q})) \Vert_{\mathbb{L}^2_{\eta}}\\
			& \le \left\{ \int_{\mathbb{R}^{d}} \int_{0}^{1}\left| \Big[\nabla u(\eta + \lambda (\tilde{B} - B)B^{-1}\eta) \cdot (\tilde{B} - B)B^{-1}\eta \Big]\right|^{2}d\lambda    d\eta \right\}^{1/2}\\
			& \le C_4(T,d) \Vert \eta \nabla u \Vert_{\mathbb{L}^2_{\eta}}E_2.
		\end{align*}
		Now we have finished the proof of the lemma. 
	\end{proof}
	
	Then, supposing
	\begin{align}\label{eqn:err_H1moment}
		\Vert \eta \nabla w \Vert_{\mathbb{L}^{2}_{\eta}} \le C_5(T,d) \varepsilon^{-d/4}, \quad \forall t \in [0,T],
	\end{align}
	the norm appeared in second part in \eqref{eqn:second_part} can be estimated as follows by \Cref{lem:estimate_moment} and argument of density,
	\begin{align}
		C^{\varepsilon}(t) \Vert w(\eta &+ \varepsilon^{-1/2}\tilde{B}(q-\tilde{q})) - w(\tilde{B}B^{-1}\eta +\varepsilon^{-1/2}\tilde{B}(q-\tilde{q})) \Vert_{\mathbb{L}^2_{\eta}} \notag\\
		& \le \frac{C_6(T,d)\exp(-\tilde{\gamma}_{I}/\varepsilon)}{\det(\alpha_{I})} E_2.
	\end{align}
	Combining the above arguments and estimate \eqref{eqn:THM1_est1}, we have, by \eqref{eqn:second_part},
	\begin{align}
		\Vert w(\eta(x))\exp\{ (\iu/\varepsilon)\tilde{\theta}(x) \}- w(\tilde{\eta}(x))\exp\{(\iu/\varepsilon)\tilde{\theta}(x)\} \Vert_{\mathbb{L}_x^2} \le \frac{C_7(T,d)\exp(-\tilde{\gamma}_{I}/\varepsilon)}{\det(\alpha_{I})} (\varepsilon^{-1/2} + 1) E_2.
	\end{align}
	\par 
	Now we investigate the first part of the inequality in \eqref{eqn:err_GWPT}, where phase error is dominant. First, notice that the phase error can be expanded as follows:
	\begin{gather} \label{eqn:decompose_phase_error}
		\begin{aligned}
			|\exp\{(\iu/\varepsilon)\theta\}-\exp\{(\iu/\varepsilon)\tilde{\theta}\}|^2 &= (\exp(-\gamma_{I}/\varepsilon)-\exp(-\tilde{\gamma}_I/\varepsilon))^2
			\\&\quad+ \exp[-(\gamma_{I}+\tilde{\gamma}_{I})/\varepsilon]\{2-\exp[(\iu/\varepsilon)(\theta_{R}-\tilde{\theta}_{R})]-\exp[(\iu/\varepsilon)(\tilde{\theta}_{R}-\theta_{R})]\}.
		\end{aligned}
	\end{gather}%
	where $ \exp[(\iu/\varepsilon)(\tilde{\theta}_{R}-\theta_{R})] $ is given by a very long formula:
	\begin{gather}\label{eqn:phase}
		\begin{aligned}
			\exp[(\iu/\varepsilon)(\tilde{\theta}_{R}-\theta_{R})] &= \exp\bigg\{ \frac{\iu}{\varepsilon}\Big[ (x-q)^{T}\alpha_{R}(q-\tilde{q}) + (q-\tilde{q})^{T}\alpha_{R}(x-q)  + (q-\tilde{q})^{T}\alpha_{R} (q-\tilde{q}) \\
			&\quad+ (x-q)^{T}(\tilde{\alpha}_{R}-\alpha_{R})(x-
			q)+(q-\tilde{q})^{T}(\tilde{\alpha}_R-\alpha_{R})(x-q)\\
			&\quad+(x-q)^{T}(\tilde{\alpha}_R-\alpha_R)(q-\tilde{q})+(q-\tilde{q})^{T}(\tilde{\alpha}_R-\alpha_R)(q-\tilde{q}) \\
			&\quad + (\tilde{p}-p)^{T}(x-q)+p^{T}(q-\tilde{q})+ (\tilde{p}-p)^{T}(q-\tilde{q}) + ( \gamma_{R}-\tilde{\gamma}_{R} 
			) \Big] \bigg\}.	
		\end{aligned}
	\end{gather}
	Since one can write 
	\begin{align*}
		\exp[-(\gamma_{I}+&\tilde{\gamma}_{I})/\varepsilon]\{2-\exp[(\iu/\varepsilon)(\theta-\tilde{\theta})]-\exp[(\iu/\varepsilon)(\tilde{\theta}-\theta)]\} \\
		&= \exp(-(\gamma_{I}+\tilde{\gamma}_{I})/\varepsilon)( 2 - 2\cos[(\iu/\varepsilon)(\theta_{R}-\tilde{\theta}_{R}))]\\
		&\le \frac{1}{\varepsilon^2}\exp(-(\gamma_{I}+\tilde{\gamma}_{I})/\varepsilon) |\theta_{R} - \tilde{\theta}_{R}|^2.
	\end{align*} 
	By \eqref{eqn:decompose_phase_error} and triangle inequality, one has
	\begin{align*}
		\Vert w(\eta(x))\exp\{(\iu/\varepsilon)\theta\} -w(\eta(x))\exp\{(\iu/\varepsilon)\tilde{\theta}\}\Vert_{\mathbb{L}_x^2} 
		&\le \Vert w  |\exp(-\gamma_{I}/\varepsilon) -\exp(-\tilde{\gamma}_{I}/\varepsilon)|\Vert_{\mathbb{L}_x^2}\\ 
		&+ \frac{1}{\varepsilon}\Vert w \exp(-(\gamma_{I}+\tilde{\gamma}_{I})/\varepsilon)|\theta_R - \tilde{\theta}_R|\Vert_{\mathbb{L}_x^2}.
	\end{align*}
	Combining \eqref{eqn:phase} and supposing $ E_2 $ is sufficiently small, we have,  
	\begin{align}
		|\theta_{R} - \tilde{\theta}_{R}| \le C_8(\alpha_R,q,p)\big(E_2(x-q)^{T}(x-q) + E_2|x-q|+ E_2\big).
	\end{align}
	Here $ C_8(\alpha_{R},q,p) $ is a function of the GWPT parameters $ \alpha_R,q,p $ and in turn is a function of time. This estimate follows from \eqref{eqn:phase} by investigating the terms and omitting the higher order term in $ E_2 $. Therefore, by formal calculation, we have
	\begin{align*}
		\Vert w(\eta(x))\exp\{(\iu/\varepsilon)\theta\} &-w(\eta(x))\exp\{(\iu/\varepsilon)\tilde{\theta}\}\Vert_{\mathbb{L}_x^2} \\
		& \le \big|  \exp(-\gamma_{I}/\varepsilon) -\exp(-\tilde{\gamma}_{I}/\varepsilon)\big| \sqrt[4]{\frac{\varepsilon^{d}}{\det(\alpha_{I})}} \Vert w \Vert_{\mathbb{L}^2_{\eta}}\\
		&\quad + C_9(\alpha_R,q,p)\exp[-(\tilde{\gamma}_{I}+\gamma_I)/(2\varepsilon )](E_2/\varepsilon)\Big(\Vert |x-q|^2w(\eta(x)) \Vert_{\mathbb{L}^2_{x}}\\
		&\quad+\Vert |x-q|w(\eta(x)) \Vert_{\mathbb{L}^2_{x}} + \Vert w(\eta(x)) \Vert_{\mathbb{L}^2_{x}}\Big).
	\end{align*}
	where $ C_9(\alpha_{R},q,p) $ is a time-dependent constant. Let $ M_4 $ be the fourth moment of $ \psi $ defined by 
	\begin{equation}\label{eqn:moment}
		M_4 \triangleq  \max_{t\in[0,T]} \langle \psi(t), |x-q(t)|^4\psi(t)\rangle.
	\end{equation}
	Then we have, by Cauchy's inequality:
	\begin{align*}
		\Vert w(\eta(x))\exp\{(\iu/\varepsilon)\theta\} &-w(\eta(x))\exp\{(\iu/\varepsilon)\tilde{\theta}\}\Vert_{\mathbb{L}_x^2} \\
		&\le \big|  \exp(-\gamma_{I}/\varepsilon) -\exp(-\tilde{\gamma}_{I}/\varepsilon)\big| \sqrt[4]{\frac{\varepsilon^{d}}{\det(\alpha_{I})}} \Vert w \Vert_{\mathbb{L}^2_{\eta}}\\
		&\quad + C_9(\alpha_R,q,p)\exp[-(\tilde{\gamma}_{I}-\gamma_I)/(2\varepsilon )](E_2/\varepsilon)(\sqrt{M_4}+\sqrt[4]{M_4} + 1)
	\end{align*}
	The $ \mathbb{L}^2(\mathbb{R}^d) $ norm of $ w $ is a constant $ \varepsilon^{-d/4} $ supposing the initial wave function is normalized. Notice another trivial inequality $ 2\sqrt[4]{M_4} \le \sqrt{M_4} + 1 $, the estimate becomes:
	\begin{align*}
		\Vert w(\eta(x))\exp\{(\iu/\varepsilon)\theta\} -w(\eta(x))\exp\{(\iu/\varepsilon)\tilde{\theta}\}\Vert_{\mathbb{L}_{x}^2}&\le \big|  \exp(-\gamma_{I}/\varepsilon) -\exp(-\tilde{\gamma}_{I}/\varepsilon)\big| \det(\alpha_{I})^{-1/4} \\
		&\quad + C_{10}(\alpha_{R},q,p)\exp(-(\tilde{\gamma}_{I}-\gamma_I)/\varepsilon )(E_2/\varepsilon)(\sqrt{M_4} + 1).
	\end{align*}
	Assuming \eqref{eqn:err_ODE} and $E_2 =  \mathcal{O}(\varepsilon) $, it can be guaranteed that:
	\begin{align}
		|\exp(-(\gamma_{I}-\tilde{\gamma}_{I})/\varepsilon)-1|&\le D_1E_2/\varepsilon,\label{eqn:err_phase1}\\
		\exp(-(\tilde{\gamma}_{I}-\gamma_{I})/\varepsilon) &\le D_2. \label{eqn:err_phase2}
	\end{align} 
	where $ D_1,D_2 $ are constants independent of $ \varepsilon $, but may depend on $ T $. Now combining this statement and \Cref{lem:clean}, we have
	\begin{theorem}\label{thm:GWPT}
		Let $ q,p,\gamma,\alpha,B $ be the exact solution to the ordinary differential equations \eqref{eqn:ODE} and $ w $ be the exact solution to the $ w $ equation \eqref{eqn:wEqn}. Furthermore suppose the $ w $ equation and the GWPT parameters are approximated by $ \tilde{w} $ and $ \tilde{q},\tilde{p},\tilde{\gamma},\tilde{\alpha},\tilde{B} $ with error as follows
		\begin{gather*}
			\max_{t\in[0,T]}\Vert w-\tilde{w} \Vert_{\mathbb{L}^2_{\eta}}\le E_1\varepsilon^{-d/4},\\
			\max_{t\in [0,T]}\max\big(\Vert q-\tilde{q} \Vert_{\infty},\Vert p - \tilde{p} \Vert_{\infty},| \gamma-\tilde{\gamma} |,\Vert \alpha-\tilde{\alpha} \Vert_{\infty},\Vert B-\tilde{B}\Vert_{\infty} \big)\le E_2 .
		\end{gather*}
		If for $ t\in[0,T] $ the fourth moment is finite
		$$ M_4 = \langle \psi, |x-q|^4\psi \rangle < +\infty,\quad t\in[0,T]. $$ 
		and condition \eqref{eqn:err_H1norm}, \eqref{eqn:err_H1moment}, \eqref{eqn:err_phase1} and \eqref{eqn:err_phase2} hold. Further suppose $ E_2\sim \mathcal{O}(\varepsilon) $, then the following estimate holds
		\begin{align}\label{eqn:GWPT_estimate}
			\Vert\psi-\tilde{\psi}\Vert_{\mathbb{L}_{x}^2}\le A_1E_1 + A_2E_2(1+\varepsilon^{-1/2}) +A_{3}(\alpha_{R},q,p)(E_2/\varepsilon)(\sqrt{M_4}+1). 
		\end{align}
		where $ A_{3}(\alpha_{R},q,p) $ is a function of $\alpha_{R},q,p $, and therefore depends on time. $ A_1,A_2 $ are positive constants which are only dependent on $ T $ and dimensionality $ d $.
	\end{theorem}
	\begin{remark}
		Notice when $ \varepsilon $ is not small, the estimates in \eqref{eqn:err_phase1} and \eqref{eqn:err_phase2} also holds. So the previous theorem is still valid even $ \varepsilon $ is not small. Practically, this theorem tells us the error of solving the GWPT parameters should not exceed $ \mathcal{O}(\varepsilon) $ when $ \varepsilon $ is small.When $ \varepsilon $ is large, we only have to make sure that $ E_2 $ is small. However, this algorithm is not efficient as traditional solvers when $\varepsilon\sim \mathcal{O}(1)$.
	\end{remark}
	\begin{remark}
		Heuristically, the estimate \eqref{eqn:err_H1norm} and \eqref{eqn:err_H1moment} hold since the $ \mathcal{O}(\sqrt{\varepsilon}) $ perturbation will not affect the solution too much. Still, we provide a rigorous proof for the two estimates in the appendix. See Appendix A for details.
	\end{remark}
	
	\subsection{The Error of Solving \eqref{eqn:wEqn} Using GWPT+HWP, 1d case}
	Theorem \ref{thm:GWPT} tells us that we can control the error when solving the semi-classical Schr\"odinger equation \eqref{eqn:SEqn} with the GWPT. The error of ODE solvers can be given by standard analysis, while the error of solving \eqref{eqn:wEqn},  \eqref{eqn:potential} with HWP requires some analysis, which will be performed in this section. \par 
	An available framework to prove the spectral convergence is given in \cite{Pasciak1980}. To apply this framework, we have to give approximating property w.r.t. the number of some certain basis, e.g., Fourier modes in \cite{Pasciak1980}. We will present an approximation theorem of semi-classical wave-packets in section \ref{subsubsec:approxi}. \par 
	Then, combining the above arguments, we will prove the a priori error bound in solving the $ w $ equation \eqref{eqn:wEqn} by the HWP method.
	\subsubsection{General formula of $ P_h,Q_h $ in 1d}
	As we mentioned in \Cref{subsec:formulation}, HWP parameters $ Q_h,P_h $ satisfy the following sets of differential equations in 1d, where scalars are commutative,
	\begin{equation}\label{eqn:QP}
		\begin{split}
			\dot{Q}_{h}&=\alpha_{I}P_{h},\\
			\dot{P}_{h}&=-4\alpha_{I}Q_{h}.
		\end{split}
	\end{equation}
	Furthermore, one notice that $ Q_h,P_h $ satisfy the relation \eqref{eqn:symp_relation2}, which in 1d becomes
	\begin{equation}\label{sym_relation}
		\overline{Q}_hP_h - \overline{P}_hQ_h = 2\iu.
	\end{equation}
	From \eqref{eqn:HWP_ini}, one can derive the initial value for \eqref{eqn:QP},
	\begin{equation}\label{eqn:ini_val_1d}
		Q_h(0) = \frac{1}{\sqrt{2}},\quad P_h(0) = \iu\sqrt{2}.
	\end{equation} 
	Now it is well known that
	\begin{align}\label{eqn:ODE_sol}
		Q_h &= \frac{1}{\sqrt{2}}\exp\bigg( 2\iu\int_{0}^{t}\alpha_{I}(s)ds \bigg),\quad P_h = \sqrt{2}\iu\exp\bigg( 2\iu\int_{0}^{t}\alpha_{I}(s)ds \bigg).
	\end{align}
	This turns out to be a very good property, which guarantees the semi-classical wave-packets given by \eqref{eqn:first_mode} and recurrence relations have fixed width.

	\subsubsection{Approximation with wave-packets after the GWPT}\label{subsubsec:approxi}
	Now we can state and prove the approximating properties of semi-classical wave-packets after the GWPT. The key observation is that Hagedorn wave-packets in 1d are rescaled Hermite functions.\citep[see][]{Hagedorn1998}\par
	And it is indeed proved in 
	\cite{Hagedorn1998} that for all $ k\in\mathbb{N} $, Hagedorn's wave-packet can be explicitly represented by
	\begin{align}\label{Hage_wp}
		\varphi_{n}^{\delta}[0,0,Q_h,P_h](x) = 2^{-n/2}(n!)^{-1/2}F_{n}(Q_h,\delta,x)\varphi_0^{\delta}[0,0,Q_h,P_h](x),
	\end{align}
	where
	\begin{align}
		F_{n}(Q_h,\delta,x) &= Q_{h}^{-n/2}(\overline{Q}_h)^{n/2}H_{n}(\delta^{-1/2}|Q_h|^{-1}x),\\
		H_{n}(x) &= \eu^{x^2}\left( -\frac{\partial}{\partial x} \right)^{n}\eu^{-x^2},\\
		\varphi_0^{\delta}[0,0,Q_h,P_h](x) &= (\pi\delta)^{-1/4}Q_h^{-1/2}\exp\left( \frac{\iu}{2\delta}P_hQ_h^{-1}x^2 \right).
	\end{align}\par
	So we can expect some approximations similar to the ones appearing in the theorems given in \cite{Shen2011}. The final goal is to prove similar properties of Hagedorn's wave-packets. In the following part, we will take the weight to be $ \omega = \exp(-|Q_h|^{-2}x^2) $ and $ F_n(x) = F_n(Q_h,1,x) $. We first review some basic properties of Hermite functions and Hermite polynomials, whose proof can be found in, e.g., section 7.2.1 of \cite{Shen2011}.
	\begin{lemma}
		The following equalities hold:
		\begin{align}
			\partial^{k}_xF_n(x) &= \frac{2^kn!}{Q_h^k(n-k)!}F_{n-k},\\
			\int_{\mathbb{R}}H_{m}H_{n}\exp(-x^2)dx &= \sqrt{\pi}2^{n}n!\delta_{mn},  \\
			\int_{\mathbb{R}}\partial_x^kF_{m}(x)\overline{\partial_x^{k}F}_{n}(x)\exp(-|Q_h|^{-2}x^2)dx &=
			h_{n,k} \delta_{mn},
		\end{align}
		where $ \delta_{mn} $ is the Kronecker symbol, and 
		\[
		h_{n,k} = \frac{2^kn!}{|Q_h|^{2k-1}(n-k)!}\sqrt{\pi}2^n n!
		\]
	\end{lemma}
	To state the first theorem, we will present the definition of weighted $ \mathbb{L}^2_{x,\omega} $ spaces. 
	\begin{defini}
		The space of weighted $ \mathbb{L}_x^2 $ functions, denoted by $ \mathbb{L}^2_{x,\omega} $ is given by
		\begin{equation}\label{eqn:weighted_norm}
			\mathbb{L}^2_{x,\omega} \triangleq \left\{ f\in \mathbb{L}^{1}_{loc}: \int_{\mathbb{R}}|f|^2\omega dx<\infty \right\},
		\end{equation}
		where $ \omega> 0 $ is the weight function and $ x $ denotes the integral variable. It is a Hilbert space whose norm is given by 
		\begin{equation}
			\Vert  f \Vert^{2}_{\mathbb{L}^{2}_{x,\omega}}\triangleq \int_{\mathbb{R}} |f|^{2}\omega(x) dx.
		\end{equation} 
		Similarly, we can define the weighted Sobolev spaces by 
		\begin{align}
			\mathbb{H}^{m}_{x,\omega}\triangleq  \left\{ f \in \mathbb{L}^{1}_{loc}: \partial_x^{\beta}f \in \mathbb{L}^{2}_{x,\omega},\; \forall|\beta|\le m, \beta \in \mathbb{N}^{d} \right\}.
		\end{align}
		And its corresponding norm $ \Vert \cdot \Vert_{\mathbb{H}^{m}_{x,\omega}} $ is given by 
		\begin{align}
			\Vert f \Vert_{\mathbb{H}^{m}_{x,\omega}} \triangleq \sum_{|\beta|\le m} \Vert \partial_x^{\beta} f \Vert_{\mathbb{L}^{2}_{x,\omega}}.
		\end{align}
	\end{defini} 
	
	\begin{theorem}\label{thm:Weighted_Estimate}
		For any given $ u\in \mathcal{S}(\mathbb{R}) $ with $ 0\le l\le m \le N+1 $. Then we have:
		\begin{equation}\label{thm1}
			\Vert \partial_x^{l}(\Pi_N u - u) \Vert_{\mathbb{L}^{2}_{x,\omega}} \le \frac{|Q_h|^{m-l}}{2^{(m-l)/2}}\sqrt{\frac{(N-m+1)!}{(N-l+1)!}}\Vert \partial^{m}_{x}u \Vert_{\mathbb{L}^{2}_{x,\omega}}.
		\end{equation}
		Here $ \Pi_{N} $ is the projection onto the space spanned by the first $ N+1 $ Hermite polynomials, i.e.  
		\begin{equation}\label{eqn:projection}
			\int_{\mathbb{R}} (u-\Pi_{N}u)v_N\omega dx = 0,\quad \forall v_N\in P_N = \operatorname{span}\big\{H_i(|Q_h|^{-1}x)\big\}_{i=0}^{N}.
		\end{equation}
		and the weighted $ \mathbb{L}_x^2 $ norm is given by \eqref{eqn:weighted_norm}. $\mathcal{S}(\mathbb{R}) $ denotes the Schwartz space.
	\end{theorem}
	\begin{proof}
		It is obvious that for $ u\in\mathbb{L}^{2}_{x,\omega} $, 
		$$ u(x) = \sum_{n=0}^{\infty}\hat{u}_n F_n(x), \quad \text{with } \hat{u}_n =  \frac{\langle u,F_n \rangle_{x,\omega}}{\langle F_n,F_n \rangle_{x,\omega}}. $$
		Then one has
		\begin{align*}
			\Vert \partial_x^l (\Pi_Nu - u)\Vert_{\mathbb{L}^{2}_{x,\omega}}^2 &= \sum_{n=N+1}^{\infty}h_{n,l}|\hat{u}_n|^2\\
			&\le\max_{n\ge N+1}\left\{ \frac{h_{n,l}}{h_{n,m}} \right\} \sum_{n=N+1}^{\infty}h_{n,m}|\hat{u}_n|^2.
		\end{align*}
		where $ \hat{u}_n $ are the expansion coefficients. 
		By the following fact, 
		\begin{align}
			\max_{n \ge N+1}\left\{ \frac{|Q_h|^{2m-2l}}{2^{m-l}}\frac{(n-m)!}{(n-l)!} \right\}& = \frac{|Q_h|^{2m-2l}}{2^{m-l}}\frac{1}{(N-l+1)(N-l)\cdots(N-m+2)}.
		\end{align}
		We have, 
		\begin{align}
			\Vert \partial_x^l (\Pi_Nu - u)\Vert_{\mathbb{L}^{2}_{x,\omega}}^2 \le \frac{h_{N+1,l}}{h_{N+1,m}}\Vert \partial_x^mu \Vert_{\mathbb{L}^{2}_{x,\omega}}^2. 
		\end{align}
	\end{proof}
	
	From the estimates for the approximation error in weighted norm, we can give the error from the raising operator. First we will give a useful lemma representing the lowering operator in 1d.
	\begin{lemma}\label{lem:Lowering_explicit}
		For any $ u\in \mathcal{S}(\mathbb{R}) $ and fixed positive integer $ l $. Suppose further that $ p_h = q_h = 0 $, we have
		\begin{align}\label{commutator}
			\exp\left(\frac{\iu}{2}P_hQ_h^{-1}x^2\right)\partial_x^{l} \left[ u\exp\left(-\frac{\iu}{2}P_hQ_h^{-1}x^2\right) \right] = 2^{l/2}Q_{h}^{-l/2}\mathcal{A}^{l}u.
		\end{align}
		where $ \mathcal{A} $ is the lowering operator given in \eqref{eqn:lowering_op} when $ \delta = 1 $.
	\end{lemma}
	\begin{proof}
		The proof is given by induction on $l$. For  $ l=1 $ the formula can be verified  by direct calculation. Now suppose the equality \eqref{commutator} holds when $ l= n-1 $. By substituting $ u $ for $ 2^{(n-1)/2}Q_h^{-(n-1)/2}\mathcal{A}^{n-1}u $ and again apply the above equality when $ l=1 $, we have:
		\begin{align*}
			\exp\left( \frac{\iu}{2}P_hQ_h^{-1}x^2 \right)&\partial_{x}\left\{ \exp\left(-\frac{\iu}{2}P_hQ_h^{-1}x^2\right)\exp\left(\frac{\iu}{2}P_hQ_h^{-1}x^2\right) \partial_{x}^{n-1}\left[ u\exp\left(-\frac{\iu}{2}P_hQ_h^{-1}x^2\right) \right] \right\} \\
			&\quad=\exp\left(\frac{\iu}{2}P_hQ_h^{-1}x^2\right)\partial_x^{n} \left[ u\exp\left(-\frac{\iu}{2}P_hQ_h^{-1}x^2\right) \right]\\
			&\quad= 2^{n/2}Q_{h}^{-n/2}\mathcal{A}\left(\mathcal{A}^{n-1}u\right).
		\end{align*}
		we can obtain the result for any positive integer $ l $.
	\end{proof}
	
	Note that \eqref{eqn:HWP_ini} and \Cref{rmk:variables} justify the special choice of $ p_h,q_h $. Now we can state and prove an approximation result related to the lowering operators:
	\begin{theorem}\label{thm:Lowering}
		For any $ u\in \mathcal{S}(\mathbb{R}), $ with $ 0\le l\le  m\le N+1 $, we have:
		\begin{equation}\label{thm2}
			\left\Vert \mathcal{A}^{l}(\hat{\Pi}_Nu-u) \right\Vert_{\mathbb{L}^{2}_{x}} \le \sqrt{\frac{(N-m+1)!}{(N-l+1)!}} \left\Vert \mathcal{A}^{m}u \right\Vert_{\mathbb{L}^{2}_{x}}.
		\end{equation}
		Here $ \hat{\Pi}_N $ is the projection operator onto the space spanned by the first $ N+1 $ Hagedorn's wave-packets
		\begin{equation}\label{eqn:projection_function}
			\hat{\Pi}_Nu \triangleq \exp\left(\frac{\iu P_hx^2}{2Q_h}\right)\Pi_N\left(u\exp\left(\frac{-iP_hx^2}{2Q_h}\right)\right) ,\quad\forall u\in\mathcal{S}(\mathbb{R}),
		\end{equation}
		and $ \mathcal{A} $ is the lowering operator given by \eqref{eqn:lowering_op} when $ \varepsilon = 1 $.
	\end{theorem}
	\begin{proof}
		
		From the above lemma, we have
		\begin{align}\label{eqn:step1}
			\Vert \mathcal{A}^{l}(\hat{\Pi}_Nu - u) \Vert_{\mathbb{L}^{2}_{x}} = \left(\frac{|Q_h|}{\sqrt{2}}\right)^l\left\Vert \exp\left(\frac{\iu}{2}P_hQ_h^{-1}x^2\right)\partial_x^l\left[ (\hat{\Pi}_Nu - u)\exp\left(-\frac{\iu}{2}P_hQ_h^{-1}x^2\right) \right] \right\Vert_{\mathbb{L}^{2}_{x}}.
		\end{align}
		Furthermore, notice another equality, which can be found in \cite{Faou2009}.
		\begin{equation}
			\Im(P_hQ_h^{-1}) = (Q_hQ_h^{\ast})^{-1} = |Q_h|^{-2},
		\end{equation}
		that leads to
		\begin{equation}
			\omega =  \exp(-|Q_h|^{-2}x^2). 
		\end{equation}\par
		Together with \eqref{eqn:step1}, which tells us
		\begin{align*}
			\Vert \mathcal{A}^{l}(\hat{\Pi}_Nu - u) \Vert_{\mathbb{L}_x^2} 
			&=\left(\frac{|Q_h|}{\sqrt{2}}\right)^l \left\Vert \partial_x^l\left[ (\hat{\Pi}_Nu - u)\exp\left(-\frac{\iu}{2}P_hQ_h^{-1}x^2\right)\right] \right\Vert_{\mathbb{L}^{2}_{x,\omega}}\\
			&=\left(\frac{|Q_h|}{\sqrt{2}}\right)^l \left\Vert \partial_x^l\left\{ (\Pi_N-\operatorname{Id}) \left[u\exp\left(-\frac{\iu P_hx^2}{2Q_h}\right)\right]\right\} \right\Vert_{\mathbb{L}^{2}_{x,\omega}}.
		\end{align*}
		Now using Theorem \ref{thm:Weighted_Estimate}, we have 
		\begin{align*}
			\Vert \mathcal{A}^{l}(\hat{\Pi}_Nu - u) \Vert_{\mathbb{L}_x^2} &\le \left(\frac{|Q_h|}{\sqrt{2}}\right)^m \sqrt{\frac{(N-m+1)!}{(N-l+1)!}}\left\Vert \partial^{m}_{x}\left[ u\exp\left( -\frac{\iu P_hx^2}{2Q_h} \right) \right] \right\Vert_{\mathbb{L}^{2}_{x,\omega}}.
		\end{align*}
		Again by \Cref{lem:Lowering_explicit}
		\begin{align*}
			\Vert \mathcal{A}^{l}(\hat{\Pi}_Nu - u) \Vert_{\mathbb{L}_x^2} &\le \sqrt{\frac{(N-m+1)!}{(N-l+1)!}}\Vert \mathcal{A}^{m}u \Vert_{\mathbb{L}^2_x}. 
		\end{align*}
		Now the conclusion immediately follows.
	\end{proof}
	
	From \Cref{thm:Weighted_Estimate} and \Cref{thm:Lowering}, we can state and prove the approximation error bound after the GWPT
	\begin{theorem}\label{thm:Final_estimate}
		Let $ \mathcal{A} $ be the lowering operator defined by \eqref{eqn:lowering_op}. If $u \in \mathcal{S}(\mathbb{R}),k=0,\,1,\,\cdots,\,m$ and $ 3\le m\le N+1 $, then
		\begin{equation}\label{thm3}
			\Vert \partial_x^{l}(\hat{\Pi}_Nu-u) \Vert_{\mathbb{L}^2_x} \le C\sqrt{\frac{(N-m+1)!}{(N-l+1)!}}\Vert \mathcal{A}^{m}u \Vert_{\mathbb{L}^2_x},\quad l=0,1,2,3.
		\end{equation}
		where $ C $ is a positive constant independent of $ m,N $ and $ u $, but dependent on $ l $. Here $ \hat{\Pi}_N $ is the projection operator defined by \eqref{eqn:projection_function}.
	\end{theorem}
	\begin{proof}
		We will prove this theorem when $ Q_h, P_h $ satisfy the general formula \eqref{eqn:ODE_sol}. Here we notice a useful fact that $\omega = \exp(-2x^2)$.  \par
		When $ l=0 $, this reduces to the previous theorem.\par
		When $ l=1 $, we can prove a lemma similar to Formula (B.36b) in \cite{Shen2011} using integration by parts
		\begin{equation}\label{eqn:B.36b}
			\Vert xu \Vert_{\mathbb{L}^{2}_{x,\omega}}^2 \le\frac{|Q_h|^2}{2-|Q_h|^2}\Vert u \Vert_{\mathbb{H}^{1}_{x,\omega}}^2 = \frac{1}{3}\Vert u \Vert_{\mathbb{H}^{1}_{x,\omega}}^2.
		\end{equation}
		And by a direct calculation
		\begin{align*}
			\partial_x (\hat{\Pi}_Nu-u) =&{} \exp\left(-x^2\right)\partial_x\left[ \Pi_N\left(u\exp\left(x^2\right)\right) -u\exp\left(x^2\right) \right]\\
			&-2x\exp\left(-x^2\right)\left[ \Pi_N\left(u\exp\left(x^2\right)\right) -u\exp\left(x^2\right)\right].
		\end{align*}
		Hence by \eqref{eqn:B.36b}, we have
		\begin{align*}
			\Vert \partial_x(\hat{\Pi}_Nu-u) \Vert_{\mathbb{L}^2_{x}} &\le\left\Vert\partial_x\big[\Pi_N(u\exp(x^2)) -u\exp(x^2)\big]\right\Vert_{\mathbb{L}^{2}_{x,\omega}} + \left\Vert 2x\left[\Pi_N(u\exp(x^2)) -u\exp(x^2)\right]\right\Vert_{\mathbb{L}^{2}_{x,\omega}} \\
			&\le C\Vert \Pi_N(u\exp(x^2)) -u\exp(x^2) \Vert_{\mathbb{H}^{1}_{x,\omega}}\\
			&= C ( \left\Vert\partial_x\big[\Pi_N(u\exp(x^2)) -u\exp(x^2)\big]\right\Vert_{\mathbb{L}^{2}_{x,\omega}} + \left\Vert \left[\Pi_N(u\exp(x^2)) -u\exp(x^2)\right]\right\Vert_{\mathbb{L}^{2}_{x,\omega}} ).
		\end{align*}
		Where $ C $ is some positive constant independent of $ N $. \par 
		By \Cref{thm:Weighted_Estimate} and the explicit formula of $ Q_h $, we can estimate the two parts in RHS respectively. 
		\begin{align*}
			\Vert \partial_x(\hat{\Pi}_Nu-u) \Vert_{\mathbb{L}^2_{x}} &\le  \frac{C}{2^{(m-1)}}\sqrt{\frac{(N-m+1)!}{N!}}\Vert \partial^{m}_{x}(u\exp(x^2)) \Vert_{\mathbb{L}^{2}_{x,\omega}}\\
			&\quad + \frac{C}{2^{m}}\sqrt{\frac{(N-m+1)!}{(N+1)!}}\Vert \partial^{m}_{x}(u\exp(x^2)) \Vert_{\mathbb{L}^{2}_{x,\omega}} \\
			&\le \frac{3C}{2^{m}}\sqrt{\frac{(N-m+1)!}{N!}}\Vert \partial^{m}_{x}(u\exp(x^2)) \Vert_{\mathbb{L}^{2}_{x,\omega}}\\
			&= \frac{3C}{2^{m}}\sqrt{\frac{(N-m+1)!}{N!}}\Vert \exp(-x^2)\partial^{m}_{x}(u\exp(x^2)) \Vert_{\mathbb{L}_{x}^2}.
		\end{align*}\par 
		The third equality holds from definition. Then we apply \eqref{commutator} that states 
		$$ \exp(-x^2)\partial^{m}_{x}\left[ u\exp(x^2) \right] = 2^{m}\exp\bigg(- 2\iu m\int_{0}^{t}\alpha_{I}(s)ds \bigg)\mathcal{A}^{m}u. $$
		This completes the proof when $ l=1 $.
		\par
		The case when $ l=2,3 $ can be proved in the same manner in GWPT regime.
	\end{proof}
	
	Considering the argument of density, the error  estimates of the spectral interpolations as in \Cref{thm:Final_estimate} also hold  for functions in certain Sobolev spaces. Such estimates are crucial for the error analysis for solving the $w$ equation to be presented in the next section.
	\subsubsection{Dynamics for the truncated GWPT+HWP method}
	To present the final theorem, we first introduce some notation.
	Let $ w_N(\eta,t) $ be the solution of the equation:
	\begin{align*}
		i\partial_t w_N &= \hat{\Pi}_NH(t)w_N\\
		w_N(\eta,0) &= \hat{\Pi}_{N}w_0
	\end{align*}
	where $ H(t) $ is defined in the $ w $ equation. We have to give another lemma concerning the dynamics of HWP method. 
	\begin{lemma}\label{GWPT+HWPeqn}
		Let $ \mathcal{K} = \{0,1,2,\cdots ,N \} $ and suppose the GWPT parameters $( q,p,\gamma,\alpha_{R},\alpha_{I},B)$, HWP parameters $( Q_h,P_h )$ and the coefficients $( \{c_{k}\}_{k\in \mathcal{K}_1} )$ are solved exactly. Then the approximate solution $ \tilde{w} $ of $ w $ equation, where $ \tilde{w} $ is given by
		$$ \tilde{w}(\eta,t) = \sum_{k\in\mathcal{K}}c_k(t)\varphi^{1}_k[0,0,Q_h,P_h](\eta), $$
		satisfies
		\begin{equation}\label{trunc_GWPT+HWP}
			i\partial_t\tilde{w} = \hat{\Pi}_N\hat{H}(t)\tilde{w}. 
		\end{equation}
		Here $ \hat{\Pi}_N $ is the projection operator onto the space spanned by the first $ N+1 $ Hagedorn's wave-packets. $ \hat{H}(t) $ is the quantum Hamiltonian operator defined in \eqref{eqn:wEqn}.
		\begin{proof}
			We do this by direct calculation.
			\begin{align*}
				\iu\partial_t \tilde{w} &= \sum_{k=0}^{N}\iu\left(\partial_t c_{k}\right)\varphi_{k}^{1} + \sum_{k=0}^{N}c_{k}\iu\partial_t\varphi_{k}^{1}\\
				&= \sum_{k=0}^{N}\left(\sum_{l=0}^{N}f_{kl}c_l \right)\varphi_{k}^{1} + \sum_{k=0}^{N}c_k\hat{H}_q\varphi^{1}_{k}\\
				&= \sum_{k=0}^{N}\left(\sum_{l=0}^{N}f_{kl}c_l \right)\varphi^{1}_{k} + \sum_{k=0}^{N}c_k(2k+1)\alpha_{I}\varphi^{1}_{k}.
			\end{align*}
			The last equality holds since Hagedorn wave-packets are the eigenfunctions of the quadratic Hamiltonian $ \hat{H}_q $ whose corresponding eigenvalues are $ \alpha_{I}(2k+1) $ (see Remark 2.4 in \cite{Hagedorn1998}).
		\end{proof}
	\end{lemma}
	
	Now we can prove the spectral convergence of the GWPT+HWP method at given time $ T $:
	\begin{theorem}\label{thm:approxi}
		Let $ w=w(t) $ be the solution of the equation \eqref{eqn:wEqn} and $ \hat{\Pi}_Nw(t)-w(t) \in \mathbb{H}^{3}_{x}(\mathbb{R}) $ up to time $ T $. Suppose $ \mathcal{A}^{m}w(t)\in \mathbb{L}^{2}_{x}(\mathbb{R}) $, $ V^{(3)}\in \mathbb{L}_{x}^{\infty}(\mathbb{R}) $ and the stability condition:
		\begin{equation}\label{eqn:stability}
			\Vert \mathcal{A}^{m}w(t) \Vert_{\mathbb{L}^2_{\eta}} \le B(m,T)\Vert \mathcal{A}^{m}w(0) \Vert_{\mathbb{L}^2_{\eta}} \quad \text{for $ t\in[0,T] $.}
		\end{equation}
		Then there exist constants $ B_1(m,T),B_2(m,T) $ independent of $ w_0 $ such that
		\begin{equation}\label{eqn:thm_appro}
			\Vert w(t)-w_N(t) \Vert_{\mathbb{L}_{\eta}^2} \le B_1(m,T)\sqrt{\frac{(N-m+1)!}{(N+1)!}}\Vert \mathcal{A}^m w(0) \Vert_{\mathbb{L}_{\eta}^2} + B_2(m,T)T\frac{\varepsilon^{1/2}}{\tilde{\alpha}_I^{3/2}} \sqrt{\frac{(N-m+1)!}{(N-2)!}} \Vert \mathcal{A}^{m}w(0) \Vert_{\mathbb{L}_{\eta}^2}.
		\end{equation}
		Here $ \tilde{\alpha}_I = \min_{0\le t \le T}\alpha_{I} $ and $ V^{(3)} $ is the third derivative of the potential function.
	\end{theorem}
	
	\begin{proof}
		The idea is simply an extension of a technique introduced in the proof in Theorem 2 in \cite{Pasciak1980}. Let $ X(t)=\hat{\Pi}_Nw(t) $ and $ v = w(t)-X(t) $, where $ \hat{\Pi}_N $ is a projection operator that depends on time. We first calculate the time derivative of $ X(t) $. Here we abbreviate the operator $ \hat{H}(t),\hat{H}_q(t) $ as $ \hat{H},\hat{H}_q $.
		\begin{align*}
			\partial_t(\hat{\Pi}_N(t)w(t)) &= \partial_t \left( \sum_{k=0}^{N}\langle w,\varphi^{1}_{k}\rangle\varphi_{k}^{1} \right)\\
			&= \sum_{k=0}^{N} \langle \frac{1}{\iu}\hat{H} w,\varphi_{k}^{1} \rangle\varphi_{k}^{1}+ \sum_{k=0}^{N}\langle w, \frac{1}{\iu}\hat{H}_q\varphi_{k}^{1}\rangle\varphi_{k}^{1} + \sum_{k=0}^{N}\langle w, \varphi_{k}^{1}\rangle \frac{1}{\iu}\hat{H}_q\varphi_{k}^{1}\\
			&= \sum_{k=0}^{N} \langle \frac{1}{\iu}\hat{H} w,\varphi_{k}^{1} \rangle\varphi_{k}^{1}- \sum_{k=0}^{N}\langle\frac{1}{\iu}\hat{H}_q w, \varphi_{k}^{1}\rangle\varphi_{k}^{1} + \sum_{k=0}^{N}\langle w, \varphi_{k}^{1}\rangle \frac{1}{\iu}\hat{H}_q\varphi_{k}^{1}.
		\end{align*}
		If $ \theta\in S^{1}_N(t) $, where $ S^{1}_N(t) $ is the linear space spanned by the first $ N+1 $ Hagedorn's wave-packets $ \{\varphi_{k}\}_{k=0}^{N} $, we then have
		\begin{align*}
			\langle\partial_t(\hat{\Pi}_Nw(t)),\theta\rangle &=\langle \frac{1}{\iu}\hat{\Pi}_N\hat{H}w,\theta\rangle  + \langle \frac{1}{\iu}(\hat{H}_q\hat{\Pi}_N-\hat{\Pi}_N\hat{H}_q)w, \theta\rangle,
		\end{align*}
		that leads to
		\begin{equation}\label{consistency}
			\langle \partial_t (\hat{\Pi}_Nw(t)),\theta \rangle - \langle\frac{1}{\iu}\hat{H}\hat{\Pi}_Nw(t) ,\theta\rangle = \langle \frac{1}{\iu\varepsilon}\left( \hat{\Pi}_NU-U\hat{\Pi}_{N} \right)w(t),\theta\rangle.
		\end{equation}
		Here $ U $ denotes the non-quadratic part of the potential function defined in \eqref{eqn:potential}. Further by Lemma \ref{GWPT+HWPeqn}, we have
		\begin{equation}\label{consistency2}
			\langle \partial_t w_N(t),\theta \rangle - \langle\frac{1}{\iu}\hat{H}w_N(t) ,\theta\rangle = 0.
		\end{equation}
		Let $ e=w_N(t)-X(t)\in S_{N}(t) $, we have, by subtracting \eqref{consistency} from \eqref{consistency2} and setting $ \theta = e $: 
		\begin{align}\label{eqn:consistency_A}
			\langle\partial_t e,e
			\rangle-\langle \frac{1}{\iu}\hat{H}e, e \rangle &= \langle \frac{\sqrt{\varepsilon}}{\iu}\left( -\hat{\Pi}_NU+U\hat{\Pi}_{N} \right)w ,e\rangle
			= \langle\frac{\sqrt{\varepsilon}}{\iu} (UX - U w) ,e\rangle.
		\end{align}
		The last equality holds by the self-adjointness of the operator $ \hat{\Pi}_N $. In the same manner, we have:
		\begin{align}\label{eqn:consistency_B}
			\langle e,\partial_t e
			\rangle-\langle e ,\frac{1}{\iu}\hat{H}e \rangle = \langle e,\frac{\sqrt{\varepsilon}}{\iu}( UX - U w) \rangle.
		\end{align}
		Thus we have, adding the above two equalities \eqref{eqn:consistency_A}, \eqref{eqn:consistency_B} and notice that $ -\iu\hat{H} $ is skew-symmetric, we have
		\begin{align*}
			\partial_t\Vert e \Vert^2_{\mathbb{L}_{\eta}^2} =2\Re\langle \frac{\sqrt{\varepsilon}}{\iu} \left(UX - U w\right) ,e\rangle \le 2\sqrt{\varepsilon}\Vert e\Vert_{\mathbb{L}^2_{\eta}} \Vert U(\hat{\Pi}_Nw-w) \Vert_{\mathbb{L}_{\eta}^2}.
		\end{align*}
		We will give a lemma to prove the boundedness of $ U $, 
		\begin{lemma}\label{lem:V2_estimate}
			Let $ \hat{\Pi}_Nw-w \in \mathbb{H}^{3}(\mathbb{R}) $ up to time $ T $. And we suppose $ V^{(3)}\in \mathbb{L}^{\infty}(\mathbb{R}) $, then we have the estimate:
			\begin{equation}\label{V2estimate}
				\left\Vert U (\hat{\Pi}_Nw-w) \right\Vert_{\mathbb{L}_{\eta}^2} \le C_1\alpha_I^{-3/2}\left\Vert  \Pi_Nw\exp(\eta^2)-w\exp(\eta^2) \right\Vert_{\mathbb{H}^{3}_{\eta,\omega}}.
			\end{equation}
		\end{lemma}
		
		\begin{proof}
			We notice the Taylor expansion with integral remainder
			\begin{equation}\label{Integral_remainder}
				U = \alpha_I^{-3/2}\frac{\eta^3}{2!}\int_{0}^{1}(1-\theta)^2V^{(3)}\left(q+\theta\eta\sqrt{\frac{\varepsilon}{\alpha_{I}}}\right)d\theta.
			\end{equation}
			So we have, by the assumption that $ V^{(3)}\in\mathbb{L}^{\infty} $
			\begin{align*}
				\Vert U (\hat{\Pi}_Nw-w) \Vert_{\mathbb{L}_{\eta}^2} &\le C_1\alpha_I^{-3/2} \Vert \eta^3(\hat{\Pi}_Nw-w) \Vert_{\mathbb{L}_{\eta}^2}\\
				&= C_1\alpha_I^{-3/2} \Vert \eta^3(\Pi_{N}-\operatorname{Id})w\exp(\eta^2) \Vert_{\mathbb{L}^{2}_{\eta,\omega}}.
			\end{align*}
			where $ C_1 $ is a constant that only depends on the potential function $ V $. To give a bound for the RHS, we can apply the technique applied in \Cref{thm:Final_estimate}. Thus we have
			$$ \Vert \eta^3 (\Pi_Nw\exp(\eta^2)-w\exp(\eta^2)) \Vert_{\mathbb{L}^{2}_{\eta,\omega}} \le C_2\Vert \Pi_Nw\exp(\eta^2)-w\exp(\eta^2) \Vert_{\mathbb{H}^{3}_{\eta,\omega}}. $$
			So we have finished the proof of the lemma.
		\end{proof}
		
		Now, combining \Cref{thm:Lowering}, we can conclude that:
		\begin{align*}
			\partial_t \Vert e \Vert_{\mathbb{L}_{\eta}^2} &\le C_3 \frac{\varepsilon^{1/2}}{\alpha_I^{3/2}} \Vert \Pi_{N}w\exp(\eta^2) - w\exp(\eta^2) \Vert_{\mathbb{H}^{3}_{\eta,\omega}}\\
			&\le C_3 \frac{\varepsilon^{1/2}}{\alpha_I^{3/2}} \sum_{l=0}^{3}\Vert  \mathcal{A}^{l}(\hat{\Pi}_Nw-w)  \Vert_{\mathbb{L}_{\eta}^2} \\
			&\le  C_3\frac{\varepsilon^{1/2}}{\alpha_I^{3/2}}\sum_{l=0}^{3}\sqrt{\frac{(N-m+1)!}{(N-l+1)!}} \Vert \mathcal{A}^{m}w \Vert_{\mathbb{L}^2_{\eta}}. 
		\end{align*}
		By the stability condition \eqref{eqn:stability}, we have
		\begin{align*}
			\partial_t \Vert e \Vert_{\mathbb{L}^2_{\eta}}&\le C_4\frac{\varepsilon^{1/2}}{\alpha_I^{3/2}} \sqrt{\frac{(N-m+1)!}{(N-2)!}}\Vert \mathcal{A}^{m}w \Vert_{\mathbb{L}^2_{\eta}} \\
			&\le C_5(m,T)\frac{\varepsilon^{1/2}}{\hat{\alpha}_{I}^{3/2}} \sqrt{\frac{(N-m+1)!}{(N-2)!}} \Vert \mathcal{A}^{m}w(0) \Vert_{\mathbb{L}^2_{\eta}}.
		\end{align*}
		Here $ \hat{\alpha}_I = \min_{0\le t \le T}\alpha_{I} $ and $ C_5(m,T) $ is a constant that depends on $ m $ and $ T $. The above inequality infers
		\begin{equation}\label{eqn:err_e}
			\Vert e \Vert_{\mathbb{L}^2_{\eta}} \le C_5(m,T)T\frac{\varepsilon^{1/2}}{\hat{\alpha}_{I}^{3/2}} \sqrt{\frac{(N-m+1)!}{(N-2)!}} \Vert \mathcal{A}^{m}w(0) \Vert_{\mathbb{L}^2_{\eta}}.
		\end{equation}
		From \Cref{thm:Final_estimate}, we have:
		\begin{align*}
			\Vert w-w_N \Vert_{\mathbb{L}_{\eta}^2} &\le \Vert w-\hat{\Pi}_{N}w \Vert_{\mathbb{L}^2_{\eta}} + \Vert \hat{\Pi}_{N}w-w_N \Vert_{\mathbb{L}^2_{\eta}}\\
			& \le C_6\sqrt{\frac{(N-m+1)!}{(N+1)!}}\Vert \mathcal{A}^m w(t) \Vert_{\mathbb{L}_{\eta}^2} + e(t). 
		\end{align*}
		Again by the stability condition \eqref{eqn:stability} and \eqref{eqn:err_e},
		\begin{align*}
			\Vert w-w_N \Vert_{\mathbb{L}_{\eta}^2} \le C_7(m,T)\sqrt{\frac{(N-m+1)!}{(N+1)!}}\Vert \mathcal{A}^m w(0) \Vert_{\mathbb{L}_{\eta}^2} + C_5(m,T)T\frac{\varepsilon^{1/2}}{\tilde{\alpha}_I^{3/2}} \sqrt{\frac{(N-m+1)!}{(N-2)!}} \Vert \mathcal{A}^{m}w(0) \Vert_{\mathbb{L}^2_{\eta}}.
		\end{align*}
		Now we have finished the proof of \Cref{thm:approxi}.
	\end{proof}
	\section{Numerical Tests}\label{sec:Numerical_Test}
	In the following section, we will perform several numerical tests for 1d and 2d cases to provide immediate evidence of the effectiveness of the method, and to verify the numerical analysis.
	\subsection{Example 1, 1d case}
	In this example \citep[see][]{Russo2013}, we choose $ d=1 $ and potential $ V(x) = 1-\cos(x) $. The initial value is chosen to be
	$$ \psi_{1d}(x,0) = \left( \frac{2}{\pi\varepsilon} \right)^{1/4}\exp\left[ -\frac{(x-\pi/2)^2}{\varepsilon} \right]. $$
	where $ q_0=\pi/2,p_0 = 0,\alpha_0 = \iu,\gamma_0=0,B_0=1,Q_{h,0} = 1/\sqrt{2},P_{h,0} = \iu\sqrt{2} $. Unless otherwise specified, the reference solution is obtained by the Fourier spectral method applied to the Schr\"odinger equation (\ref{eqn:SEqn}), using the SP4 method proposed in \cite{Chin2001}, with uniform grid in $ [-\pi,\pi) $ with space step $ \Delta x_{REF} = 2\pi\varepsilon/64 $ and time step $ \Delta t_{REF}=\varepsilon/64 $. \par   
	First, we observe that our algorithm can resolve the typical three scales (see Section \ref{subsec:review_GWPT}) as shown in Figure \ref{fig:scale_resolving}.
	We can see that, at variance with the original wave function $\psi$, the wave function $ w $ is not oscillatory, neither in space nor in time. Indeed we find out that it requires much fewer grid  points on the computational domain, and a much larger  time step $\Delta t_c$, independently on  the small scale $\varepsilon$, still reaching high accuracy.
	\begin{figure}[t!]
		\centering
		\includegraphics[scale=0.5]{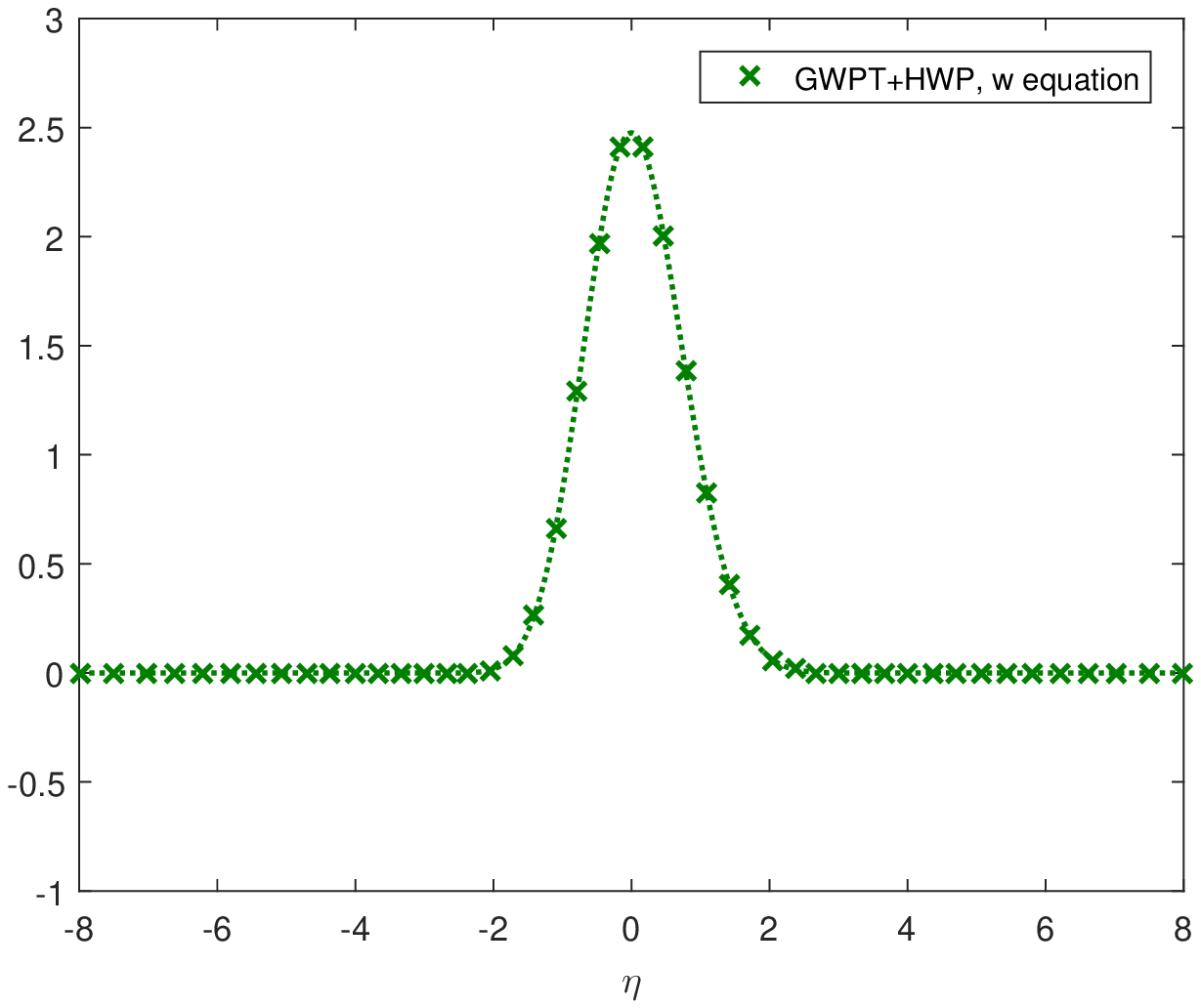}
		\hspace{-0.5cm}
		\includegraphics[scale=0.5]{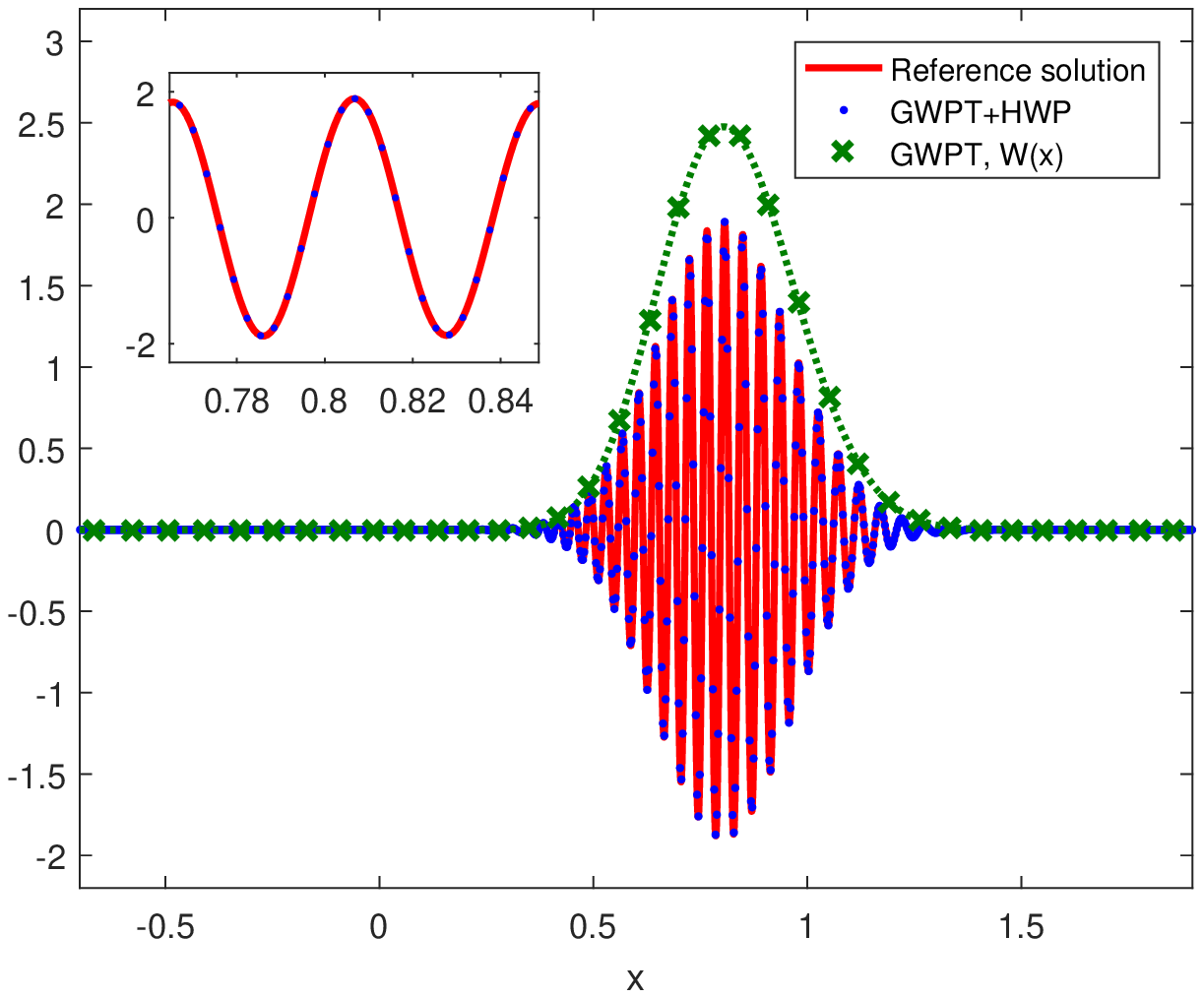}
		\caption{We choose $ \varepsilon= 1/128$, 30 wave-packets, time step $ \Delta t_c=1/64 $, $ \Delta t_{gt} = 1/1024 $ to compute the solution at $ t_f=1.25 $. 50-point Gauss-Hermite quadrature rule is used to evaluate the Galerkin matrix. Left panel: the numerical solution of the real part of $ w $ equation \eqref{eqn:wEqn} on quadrature points is plotted in $ \eta $ space at final time. The cross represent the computed solution on the grid point, while the dotted line represents the computed solution on $ \{-\pi + 2\pi (k-1)/2048\}_{k=1}^{2048} $. Right panel: the real part of the reconstructed wave function and the real part of the reference solution are plotted. The $ W $ function is also plotted on quadrature points transformed to the real space. The dark green dotted line represent the computed solution of $ W $ on $ \{-\pi + 2\pi (k-1)/2048\}_{k=1}^{2048}  $. The wave function on $ \{-\pi + 2\pi (k-1)/2048\}_{k=1}^{2048} $ is also reconstructed and is represented by blue dots. It is clear that our method requires fewer points in real space. Reference solution for $\psi$: SP4 on $ (-\pi,\pi) $ with $ \Delta t_{REF}= \varepsilon/64 $ and spatial grid size $ \Delta x_{REF} = 2\pi\varepsilon/64 $.}
		\label{fig:scale_resolving}
	\end{figure}\par
	Second, to determine how one chooses time step $ \Delta t_{c} $ and $ \Delta t_{gt} $ to achieve best approximation, we plot the $ \mathbb{L}^2 $ error for various time step $ \Delta t_{c} $ and $ \Delta t_{gt} $ with $ 30 $ wave-packets. The result is shown in Figure \ref{fig:Delta_t}. In left panel, for fixed $ \Delta t_{c} =1/128 $, we test various values of $ \Delta t_{gt} $ at final time $ t_f = 0.125 $. This numerical test verifies the fourth order convergence at finite time and in turn verifies the numerical analysis given by \eqref{eqn:GWPT_estimate}. In the right panel, we fix $ \Delta t_{gt} = 1/2048 $ and compare the $ \mathbb{L}^2 $ error for different values of $ \Delta t_c $ at final time $ t_f = 0.125 $. It is shown that when $ \Delta t_{gt} $ is small enough, $ \Delta t_c $ can be chosen independent of $ \varepsilon $.  
	\begin{figure}[htbp]
		\centering
		\includegraphics[scale=0.5]{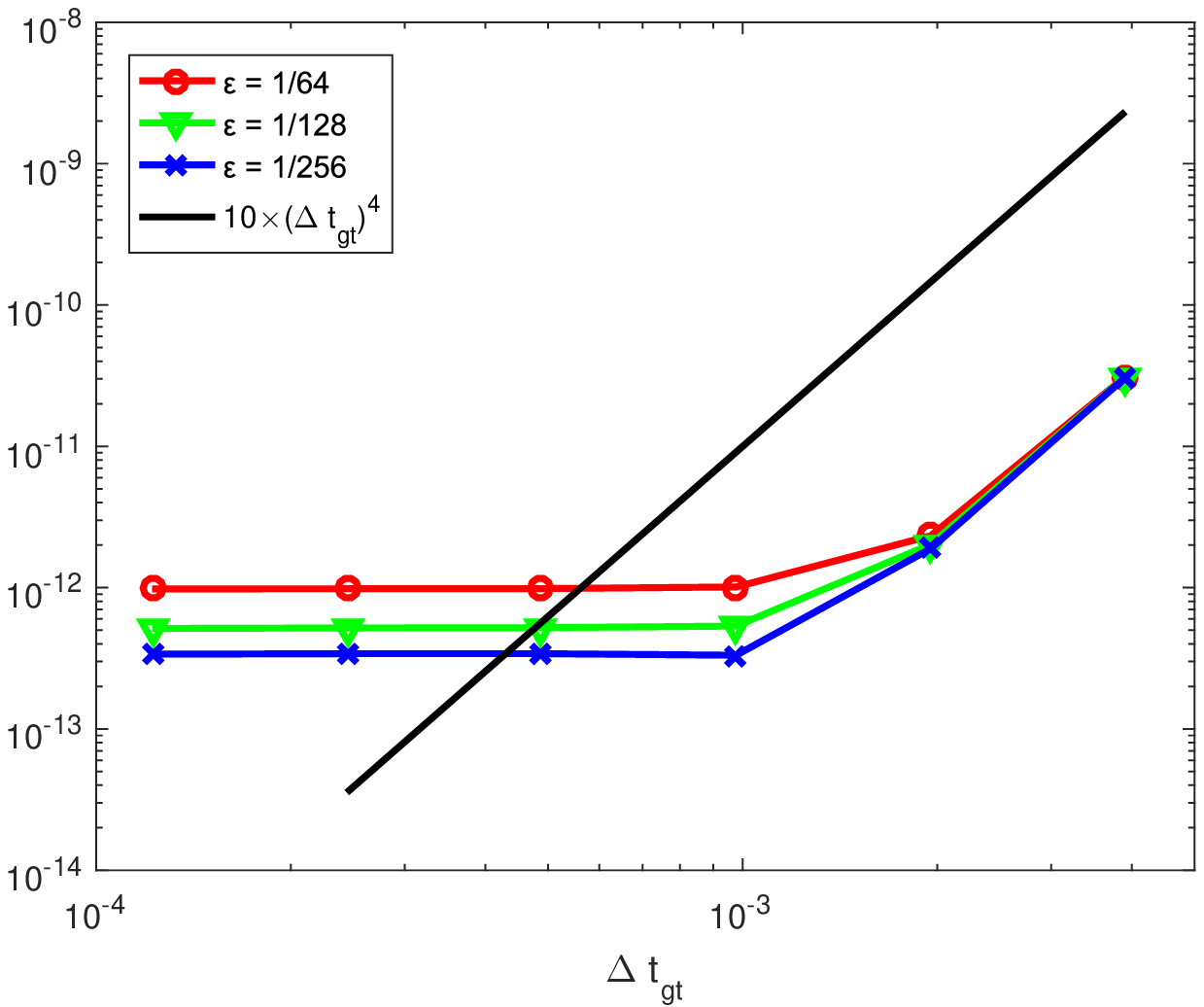}
		\hspace{-0.5cm}
		\includegraphics[scale=0.5]{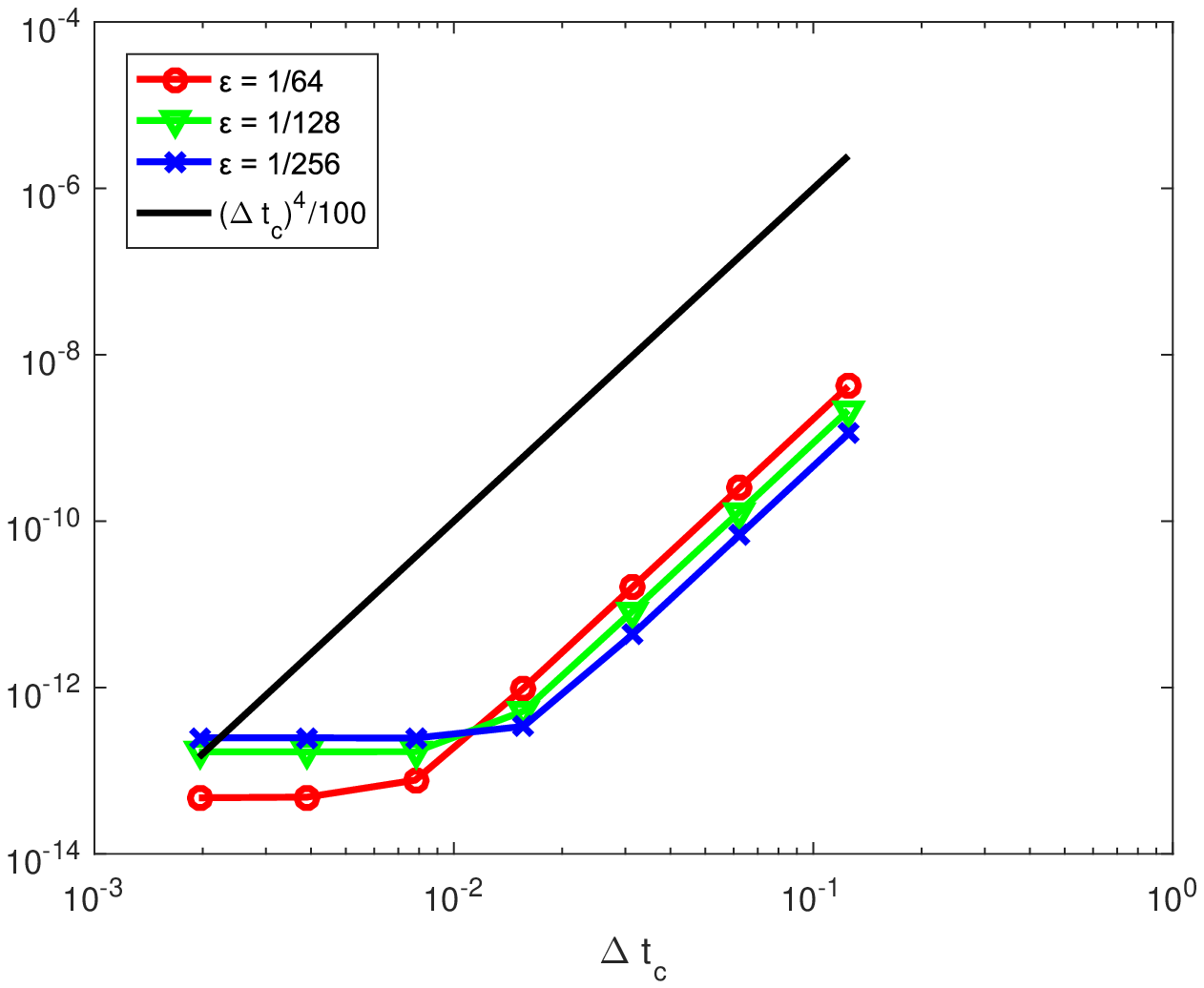}
		\caption{Left panel: We fix $ \Delta t_{c}=1/128 $ and compare the $ \mathbb{L}^2 $ error for different values of $ \Delta t_{gt} $ at final time $ t_f = 0.125 $. We choose $n = 30$ wave-packets in each test. $N_Q = 50 $-point Gauss-Hermite rule is used to evaluate the Galerkin matrix. Right panel: We fix time step satisfying $ \Delta t_{gt}=1/2048 $ and compare the $ \mathbb{L}^2 $ error for different values of $ \Delta t_{c}$ at final time $ t_f = 0.125 $. We choose 30 wave-packets in each test. $ 50 $-point Gauss-Hermite rule is used to evaluate the Galerkin matrix. Reference solution for $ \psi $: SP4 on $ (-\pi,\pi) $ with $ \Delta t_{REF}= \varepsilon/64 $ and spatial grid size $ \Delta x_{REF} = 2\pi\varepsilon/64 $}
		\label{fig:Delta_t}
	\end{figure}
	Third, we will verify the spectral convergence with respect to the number of wave-packets $ n $. We take various values of $ n $ and let the quadrature points be $N_Q = 30 $ to achieve the best approximation of Galerkin matrix. The time step is chosen to be $ \Delta t_c = 1/128 $, $ \Delta t_{gt} =1/2048 $. The result is shown in the left panel of Figure \ref{fig:Test_2}.
	\begin{figure}[htbp]
		\centering
		\includegraphics[scale=0.5]{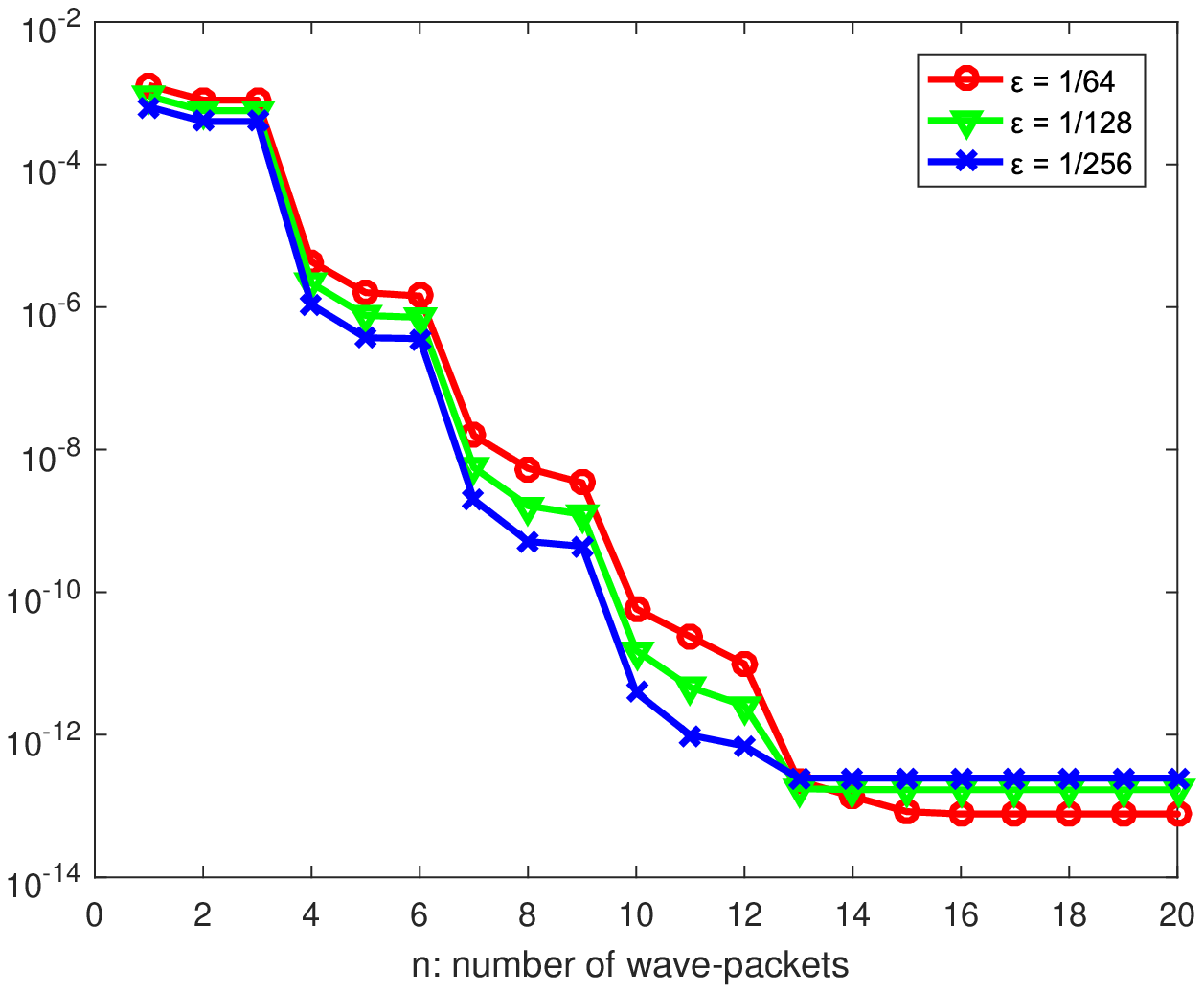}
		\hspace{-0.5cm}
		\includegraphics[scale=0.5]{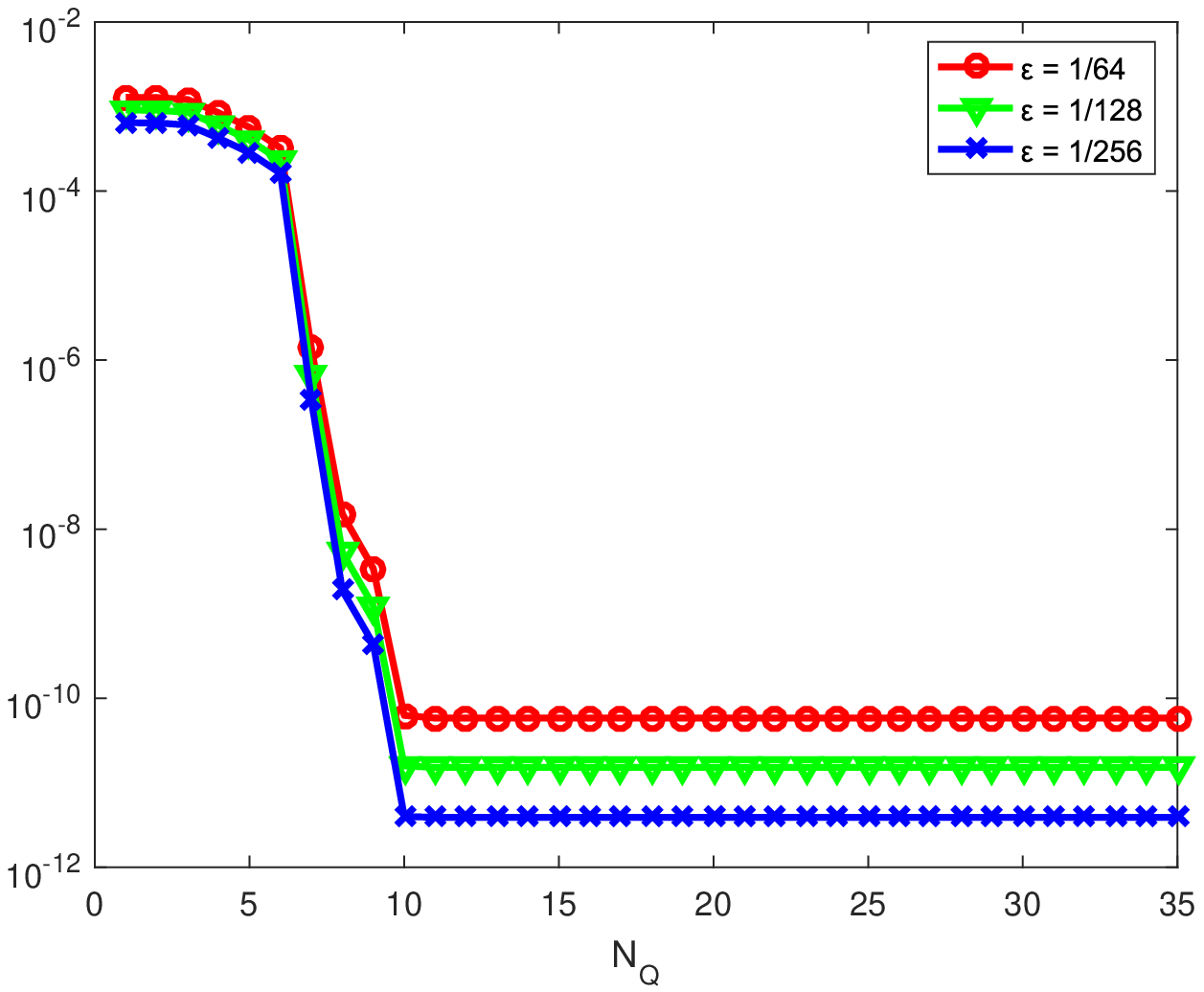}
		\caption{Left panel: For fixed time step size $ \Delta t_c=1/128 $, $ \Delta t_{gt} = 1/2048 $ and final time $ t_f = 0.125 $, we compare the $ \mathbb{L}^2 $ error of GWPT+HWP in 1d with different values of wave-packets $ n $ and $ \varepsilon $. $ 30 $-point Gauss-Hermite quadrature rule is used to evaluate the Galerkin matrix. Right panel: For fixed time step satisfying $ \Delta t_c =1/128 $, $ \Delta t_{gt} = 1/2048 $ and final time $ t_f = 0.125 $, we compare the $ \mathbb{L}^2 $ error of GWPT+HWP in 1d with different values of quadrature points $ N_Q $ and $ \varepsilon $. We choose $ 10 $ wave-packets in each test. Reference solution for $ \psi $: SP4 on $ (-\pi,\pi) $ with $ \Delta t_{REF}= \varepsilon/64 $ and spatial grid size $ \Delta x_{REF} = 2\pi\varepsilon/64 $.}
		\label{fig:Test_2}
	\end{figure}
	We observe the spectral convergence, thus the results validate our previous analysis given in \eqref{eqn:thm_appro}. The plateau appeared in the result is due to the error of the reference solution. \par
	Then we will test the relation between the number of Gauss-Hermite quadrature points and the $ \mathbb{L}^2 $ error at $ t_f=0.125 $, with time step $ \Delta t_c = 1/128 $, $ \Delta t_{gt} = 1/2048$. Here we choose $n = 10 $ wave-packets. The result is shown in the right panel of Figure \ref{fig:Test_2}. This shows the Galerkin matrix is solved with spectral accuracy with respect to the number of quadrature points $ N_{Q} $. The results also tell us how to choose the number of quadrature points properly to reduce computational cost.\par 
	In the following part, we test the relation between $ \alpha_{I} $ and the error of GWPT method. We will fix $ n=30 $ and time step $ \Delta t_c=1/128 $, $ \Delta t_{gt} =1/2048 $. The quadrature points is chosen to be $ N_{Q}=50 $ when evaluating the Galerkin matrix. The result is shown in Figure \ref{fig:Test_1}.
	\begin{figure}[htbp]
		\centering
		\includegraphics[width=10cm]{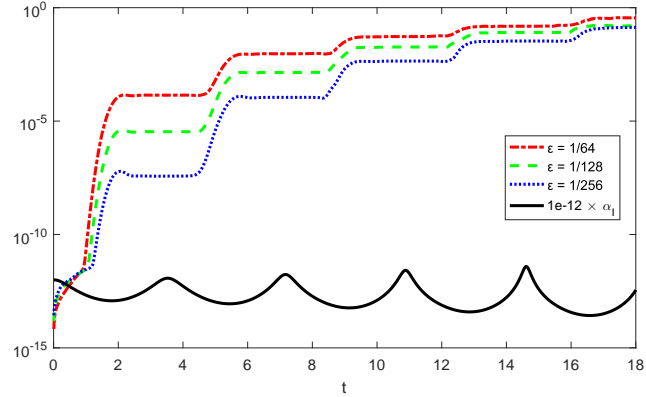}
		\caption{For fixed time step $ \Delta t_{c}= 1/128$, $ \Delta t_{gt}=1/2048 $, we compare the $ \mathbb{L}^2 $ error  with $\alpha_I$ from $ t = 0 $ to $ t=18 $. 100-point quadrature rule is used to evaluate the Galerkin matrix. Reference solution for $\psi$: SP4 on $ (-\pi,\pi) $ with $ \Delta t_{REF}= \varepsilon/64 $ and spatial grid size $ \Delta x_{REF} = 2\pi\varepsilon/64 $}
		\label{fig:Test_1}
	\end{figure}  
	As we can see from the above figure, the error has a sudden increase when $ \alpha_{I} $ is small. This behavior is typical in the GWPT, as mentioned in \cite{Russo2013,Russo2014}. This phenomenon also verifies the factor appeared in the error estimate.\par 
	Last but not least, for fixed final time $ t_f = 4 $ and time step $ \Delta t_c= 1/64 $, $ \Delta t_{gt}=1/1024 $, we compare the CPU time for various number of wave-packets. Given $ n $, the number of wave-packets, we choose $ n+5 $ Gauss-Hermite quadrature rule to evaluate the Galerkin matrix. The results are given in the following Table \ref{table:CPU1d}. 
	\begin{table}
		\centering
		\begin{tabular}{|c|c|c|c|}
			\hline
			\diagbox{$ n $ }{CPU time}{$ \varepsilon  $} &   1/64   &  1/256   &  1/1024  \\ \hline
			8                       & 0.062209 & 0.062140 & 0.061486 \\ \hline
			16                      & 0.077569 & 0.077157 & 0.076715 \\ \hline
			32                      & 0.129524 & 0.124898 & 0.124819  \\ \hline
			64                      & 0.330558 & 0.312421 & 0.325762 \\\hline
		\end{tabular} 
		\caption{For different values of $\varepsilon$, we compare the CPU time for different numbers of wave-packets $ n $ with fixed time step size $ \Delta t_c=1/64 $, $ \Delta t_{gt} = 1/1024 $ and final time $ t_f = 4 $. $ n+5 $-point Gauss-Hermite rule is applied to evaluate the Galerkin matrix.}
		\label{table:CPU1d}	
	\end{table}

	\subsection{Example 2, 2d Case} Besides testing the properties of GWPT+HWP in 2d, we will explore the possibility of applying the method to multi-dimensional cases. Note that we do not have rigorous numerical analysis for multi-dimensional cases. In this example, we choose the potential function as $ V(x) = 2-\cos(x)-\cos(y) $. The initial value is chosen to be:
	$$ \psi_{2d}(x,y,0) =  \left(\frac{2}{\pi\varepsilon}\right)^{1/2} \exp\bigg[ -\frac{(x-\pi/2)^2+(y-\pi/3)^2}{\varepsilon} \bigg]. $$
	The reference solution is computed by the SP4 method with uniform spatial collocation point on $ [-\pi,\pi)^2 $ and with spatial grid size $ \Delta x_{REF} = 2\pi\varepsilon/16 $ and time step $ \Delta t_{REF}=\varepsilon/32 $.\par
	First, we will show that the difference between $ \mathcal{K} = \mathcal{K}^{n}_{\infty} $ and $ \mathcal{K} =\mathcal{K}^{n}_{1} $, as defined in \eqref{eqn:full_grid}-\eqref{eqn:tri_grid}, is negligible. We choose $ \varepsilon=1/128 $, $ \Delta t_c = 1/128 $, $ \Delta t_{gt}=1/2048 $. The quadrature rule is chosen to be the tensor product of two $ 50 $-point Gauss-Hermite rules to achieve the best approximation to the Galerkin matrix for both cases. The result is shown in Table \ref{table:FT}. We can see that the difference in the error between $ \mathcal{K}= \mathcal{K}^{n}_{\infty} $ and $ \mathcal{K}=\mathcal{K}^{n}_{1} $ is quite small. 
	\begin{table}
		\centering
		\begin{tabular}{|c|c|c|c|c|}
			\hline
			\multicolumn{1}{|c|}{n}&\multicolumn{1}{|c|}{$ \mathcal{K}=\mathcal{K}^{n}_{\infty} $}&\multicolumn{1}{|c|}{$\#(\mathcal{K}^{n}_{\infty}) $}&\multicolumn{1}{|c|}{$ \mathcal{K}=\mathcal{K}^{n}_{1} $}&\multicolumn{1}{|c|}{$ \#(\mathcal{K}^{n}_{1}) $}\\\hline
			4   & 0.0037 &  25  & 0.0038 & 15  \\\hline 
			9   & 2.2075e-04 &  100 & 2.2182e-04 & 55  \\\hline     
			14  & 1.9813e-05 &  225 & 1.9850e-05 & 120 \\\hline
			19  & 2.3519e-06 &  400 & 2.3552e-06 & 210 \\\hline
			24  & 3.6864e-07 &  625 & 3.6904e-07 & 325 \\\hline
			29  & 6.8733e-08 &  900 & 6.8794e-08 & 465 \\\hline
		\end{tabular} 	
		\caption{For $\varepsilon=1/128$, $ \Delta t_c=1/128 $, $ \Delta t_{gt} = 1/2048$ and final time $ t_f = 2 $, we compare the $ \mathbb{L}^2 $ error of $ \mathcal{K} = \mathcal{K}^{n}_{\infty} $ with the error of $ \mathcal{K}^{n}_{1} $. The tensor product of two $ 50 $-point Gauss-Hermite rule is used to evaluate the Galerkin matrix. Reference solution for $\psi$: SP4 on $ (-\pi,\pi)^{2} $ with $ \Delta t_{REF}= \varepsilon/32 $ and spatial grid size $ \Delta x_{REF} = 2\pi\varepsilon/16 $}
		\label{table:FT}
	\end{table}  \par
	Second, we will numerically verify the spectral convergence w.r.t. $ n $ in $ \mathcal{K}_{T}^{1} $. We choose $ \Delta t_c = 1/128 $, $ \Delta t_{gt} = 1/2048 $ and final time $ t_f = 0.125 $. The quadrature rule is chosen to be the tensor product of two $ 50 $-point Gauss-Hermite rules to evaluate the Galerkin matrix. The result is shown in the left panel of Figure \ref{fig:test4}. 
	\begin{figure}[htbp]
		\centering
		\includegraphics[scale=0.5]{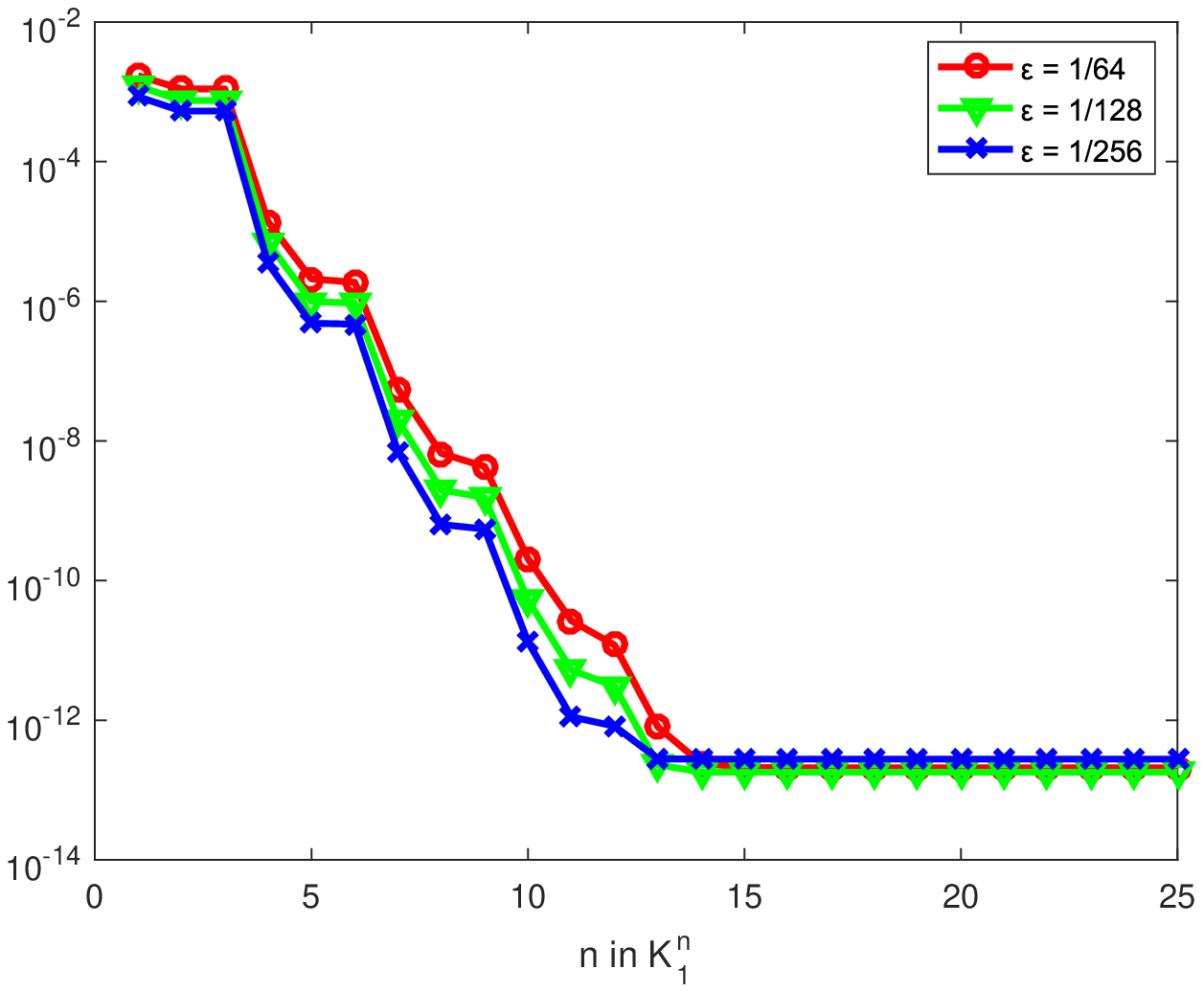}
		\hspace{-0.5cm}
		\includegraphics[scale=0.5]{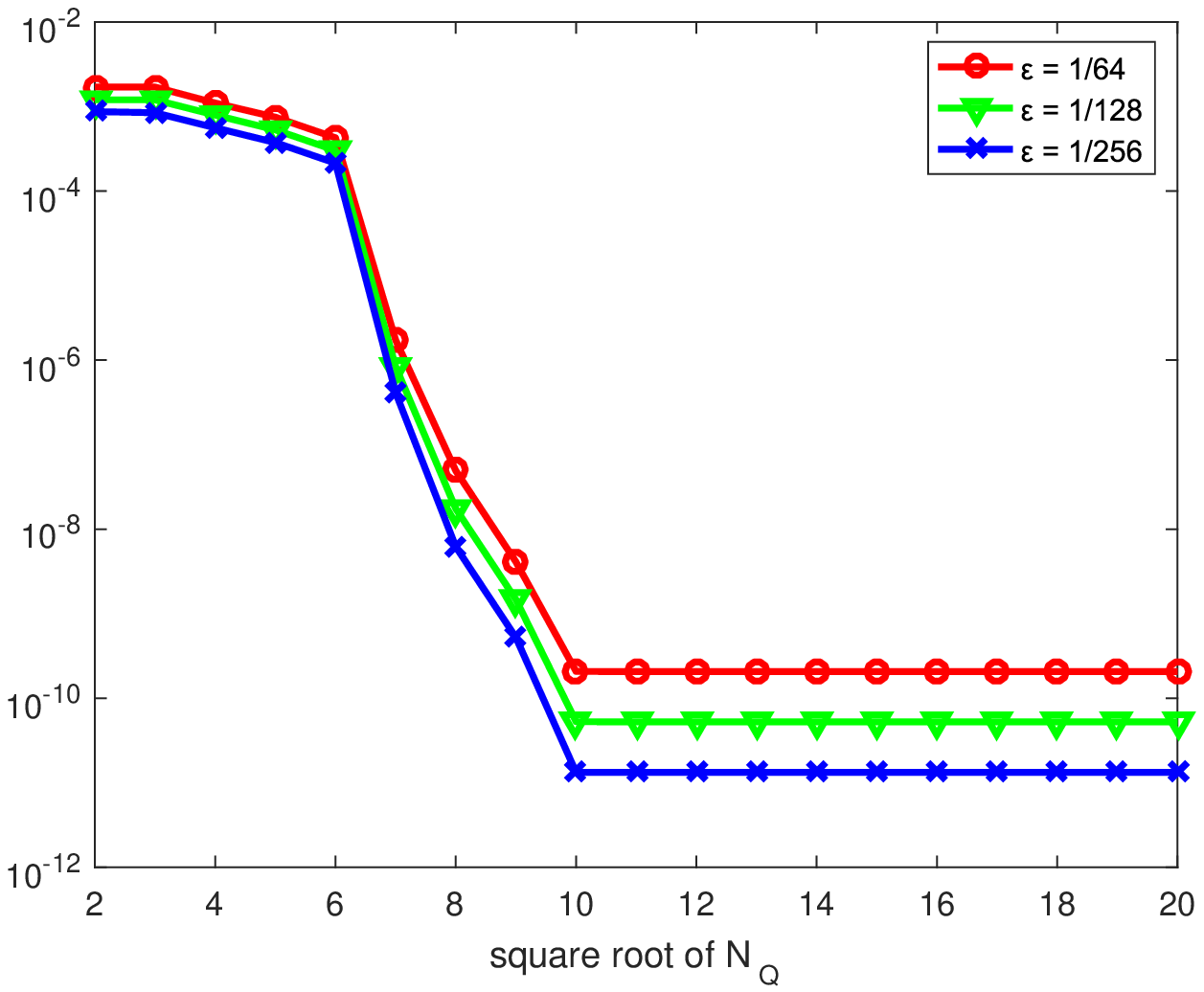}
		\label{fig:2d_NGH}
		\caption{Left panel: For fixed time step size $ \Delta t_{c} = 1/128 $, $ \Delta t_{gt}=1/2048 $ and final time $ t_f = 0.125 $, we compare the $ \mathbb{L}^2 $ error of GWPT+HWP in 2d with different values of $ n $ in $\mathcal{K}=\mathcal{K}^{n}_{1}$. The tensor product of two $ 50 $-point Gauss-Hermite rule is used to evaluate the Galerkin matrix. Right panel: For fixed time step size $ \Delta t_{c} = 1/128 $, $ \Delta t_{gt} = 1/2048 $ and final time $ t_f = 0.125 $, we compare the $ \mathbb{L}^2 $ error of GWPT+HWP in 2d. The wave-packets is $ \mathcal{K} = \mathcal{K}_{1}^{10} $ in each test. Reference solution for $\psi$: SP4 on $ (-\pi,\pi)^{2} $ with $ \Delta t_{REF}= \varepsilon/32 $ and spatial grid size $ \Delta x_{REF} = 2\pi\varepsilon/16 $.}
		\label{fig:test4}
	\end{figure}  
	\par
	Third, we will test the relation between different numbers of quadrature points and the $ \mathbb{L}^2 $ error. We choose $ \mathcal{K}=\mathcal{K}^{10}_{1} $ and $ \Delta t_c = 1/128 $, $ \Delta t_{gt} = 1/2048 $ to compare the $ \mathbb{L}^2 $ error at $ t_f=0.125 $ for various number of quadrature points $ N_Q $. The quadrature rule is chosen to be the tensor product of two $ \sqrt{N_Q} $-point Gauss-Hermite rules to evaluate the Galerkin matrix. The result is shown in the right panel of Figure \ref{fig:test4}. \par
	Last but mot least, we will also compare CPU time for different values of $ n $ in $ \mathcal{K}^{n}_{1} $. We fixed the final time $ t_f=4 $ and time step $ \Delta t_c  = 1/64 $, $ \Delta t_{gt} = 1/1024 $. We choose tensor product of two $ n+5 $ points Gauss-Hermite quadrature rule. The results are shown in Table \ref{table:2d}. The time spent is independent of the semi-classical parameter $ \varepsilon $. 
	\begin{table} %
		\centering
		\begin{tabular}{|c|c|c|c|}
			\hline
			\diagbox{$ n $ }{CPU time}{$ \varepsilon  $} &    1/64    &   1/256    &   1/1024   \\ \hline
			8                       &  1.129205  &  1.088372  &  1.091377  \\ \hline
			16                      &  6.322889  &  6.320319  &  6.310236  \\ \hline
			32                      & 102.838968 & 107.011030 & 104.107240 \\ \hline
			64                      & 744.999232 & 740.200646 & 740.016302 \\ \hline
		\end{tabular}
		\caption{For different values of $\varepsilon$, we compare the CPU time for different numbers of wave-packets $ n $ with fixed time step size $ \Delta t_c=1/64 $, $ \Delta t_{gt} = 1/1024 $ and final time $ t_f = 4 $. The tensor product of two $ n+5 $-point Gauss-Hermite rule is applied to evaluate the Galerkin matrix.}
		\label{table:2d} 	
	\end{table}
	\section{Conclusions}\label{sec:Remarks}
	In this article, we propose a novel spectral method based on the GWPT and Hagedorn's wave-packets. It circumvents the problem of artificial boundary condition of the original GWPT method, and facilitates numerical analysis. We provide a general framework for the error analysis of  GWPT based numerical methods, and prove that our proposed method has spectral convergence with respect to the number of the wave-packets in one dimension. Various numerical tests are presented to show different properties of the novel spectral method. The numerical results also validates our numerical analysis. This method does not use fast Fourier transform, which shows the possibility of parallel algorithm in moderate high dimension. Preliminary results indicate the method could be efficient in several dimensions, and can be carried out in a more efficient way, so we plan to adopt this approach for more challenging high dimensional problems.  \par 
	\section*{Acknowledgements}
	Z. Zhou is supported by the National Key R\&D Program of China, Project Number 2020YFA0712000 and NSFC grant No. 11801016, No. 12031013. Z. Zhou is also partially supported by Beijing Academy of Artificial Intelligence (BAAI). 
	G. Russo acknowledges partial support from ITN-ETN Horizon 2020 Project ModCompShock, ``Modeling and Computation of Shocks and Interfaces'', Project Reference 642768, and from the Italian Ministry of University and Research (MIUR), PRIN Project 2017 (No. 2017KKJP4X entitled “Innovative numerical methods for evolutionary partial differential equations and applications”.\\
	
	\bibliographystyle{IMANUM-BIB}
	\bibliography{IMANUM-refs}
	
	\appendix
		\section*{Appendix A.\ Proof of estimate \eqref{eqn:err_H1norm} and estimate \eqref{eqn:err_H1moment}}
		\label{app1}
		\subsection*{A.1\ Review of basic facts}
		We recall that the $ w(\eta,t) $ satisfies 
		\begin{align*}
			\iu\partial_t w&=-\frac{1}{2}\operatorname{tr}(B^{T}\nabla_{\eta}\nabla_{\eta}wB)+2\eta^{T}BB^{T}\eta w+\sqrt{\varepsilon} U(\eta;q)w\triangleq \hat{H}(t)w,\\
			w(\eta,0) &= \frac{2^{d/4}}{(\pi\varepsilon)^{d/4}} \exp ( -\eta^{T}\eta ) \triangleq w_{0}(\eta).
		\end{align*}
		where $ U(\eta;q) = \varepsilon^{-3/2} V_2(\sqrt{\varepsilon}B^{-1}\eta;q) \sim \mathcal{O}(1)$ is given by
		\begin{equation*}
			U(\eta;q) = \frac{1}{\varepsilon^{3/2}}[V(q + \sqrt{\varepsilon}B^{-1}\eta) - V(q) - \sqrt{\varepsilon} \nabla V(q)(B^{-1}\eta) - \frac{\varepsilon}{2}(B^{-1}\eta)^{T}\nabla\nabla V(q)B^{-1}\eta].
		\end{equation*} 
		Then we introduce the non-autonomous semi-group. First we consider the problem
		\begin{align}
			\iu\partial_t w_q&=-\frac{1}{2}\operatorname{tr}(B^{T}\nabla_{\eta}\nabla_{\eta}w_qB)+2\eta^{T}BB^{T}\eta w_q = \hat{H}_{q}w_q,\\
			w_q(\eta,0) &= w_{0}(\eta).
		\end{align}
		The operator $ -\iu\hat{H}_q(t) $ generates a non-autonomous semi-group defined by 
		\begin{defini}
			We fix $ T>0 $ and call a two parameter family $ (R_{t,s})_{T\ge t\ge s\ge 0} $ a \emph{continuous non-autonomous semi-group} of linear and bounded operators on Banach space $ X $ if the following conditions hold
			\begin{itemize}
				\item [(i)] $ \forall s,t\ge 0 $, the operator $ R_{t,s}: f\mapsto R_{t,s}f $ is linear and continuous on $ X $;
				\item [(ii)] $ \forall f\in X $, the map $ (t,s)\mapsto R_{t,s}f $ is continuous, where $ 0\le s\le t\le T $;
				\item [(iii)] $ \forall T\ge t\ge r \ge s \ge 0 $, $ R_{s,s} = \operatorname{Id} $ and $ R_{t,r}\circ R_{r,s} = R_{t,s} $;
				\item [(iv)] There exists $ b\in \mathbb{R} $, $ \exists M\ge 1, $
				\begin{align}
					\Vert R_{t,s} \Vert_{X}\le M\exp[b(t-s)],\quad \forall T  \ge t\ge s\ge 0.
				\end{align}
			\end{itemize}
		\end{defini}
		By the expansion of Hagedorn's wave-packets, we can prove that $ -\iu\hat{H}_{q} $ generates a continuous non-autonomous unitary semi-group. In the following part of the article, we let $ R_{t,s} $ to be the non-autonomous semi-group generated by the operator $ -\iu\hat{H}_{q} $. \par  
		If we suppose the perturbation $ U(\eta,t) $ is regular enough, we have, by Duhamel's principle:
		\begin{align}\label{eqn:Duhamel}
			w(\eta ,t) = R_{t,0}w_0(\eta) -\iu \sqrt{\varepsilon}\int_{0}^{t}R_{t,s}U(\eta;q(t))w(\eta,s)ds.
		\end{align}
		In this section, we do not intend to prove the two estimates \eqref{eqn:err_H1norm},\eqref{eqn:err_H1moment} in the most general form. So let us suppose that $ U(\eta;q(t)) $ and its gradient $ \nabla_{\eta} U(\eta;q(t)) $ are bounded uniformly in $\mathbb{L}^{\infty}$ for $ t\in [0,T] $.\par
		To prove the desired estimates, we recall the integral form of Gronwall's inequality.
		\begin{lemma}[Gronwall's inequality]\label{lem:Grownwall}
			If $\alpha,\beta\in \mathbb{L}^{\infty}_{t}([0,T])$ is non-negative and if $ u $ satisfies the integral inequality
			$$ u(t) \le \alpha(t)+\int_{0}^{t}\beta(s)u(s)ds,\quad \forall t\in [0,T], $$
			then 
			$$ u(t) \le \alpha(t)+\int_{0}^{t}\alpha(s)\beta(s)\exp\left( \int_{0}^{s}\beta(r)dr \right)ds,\quad \forall t\in [0,T]. $$
		\end{lemma}
		Finally, we will prove that the total mass of $ w $ equation up to some finite time $ T $ is bounded. 
		\begin{lemma}
			If $ U(\eta;q(t)) $ is bounded uniformly for all $ t\in[0,T] $, then we have
			\begin{align}
				\Vert w(\cdot,t) \Vert_{\mathbb{L}^{2}_{\eta}} \le C(T)\varepsilon^{-d/4},
			\end{align}
			where $ C(T) $ is a constant depending on $ T $.
		\end{lemma}
		\begin{proof}
			By Duhamel's principle, we have
			\begin{align*}
				w(\eta ,t) = R_{t,0}w_0(\eta) -\iu \sqrt{\varepsilon}\int_{0}^{t}R_{t,s}U(\eta;q(t))w(\eta,s)ds.
			\end{align*}
			Then by Minkowski's inequality, it directly follows that 
			\begin{align*}
				\Vert w(\cdot,t) \Vert_{\mathbb{L}^{2}_{\eta}} &= \Vert R_{t,0}w_{0}(\cdot) \Vert_{\mathbb{L}^{2}_{\eta}} + \sqrt{\varepsilon}\bigg\Vert \int_{0}^{t}R_{t,s}U(\cdot;q(t))w(\cdot,s)ds \bigg\Vert_{\mathbb{L}^{2}_{\eta}}\\
				&\le \Vert R_{t,0}w_{0}(\cdot) \Vert_{\mathbb{L}^{2}_{\eta}} + \sqrt{\varepsilon} \int_{0}^{t}\Vert R_{t,s}U(\cdot;q(t))w(\cdot,s)\Vert_{\mathbb{L}^{2}_{\eta}}ds .
			\end{align*}
			Since $ R_{t,s} $ is a unitary operator we have
			\begin{align*}
				\Vert R_{t,0}w_{0}(\cdot) \Vert_{\mathbb{L}^{2}_{\eta}} + \sqrt{\varepsilon} &\int_{0}^{t}\Vert R_{t,s}U(\cdot;q(t))w(\cdot,s)\Vert_{\mathbb{L}^{2}_{\eta}}ds \\&= \Vert w_{0}(\cdot) \Vert_{\mathbb{L}^{2}_{\eta}} + \sqrt{\varepsilon} \int_{0}^{t}\Vert U(\cdot;q(t))w(\cdot,s)\Vert_{\mathbb{L}^{2}_{\eta}}ds\\
				&\le \Vert w_{0}(\cdot) \Vert_{\mathbb{L}^{2}_{\eta}} + \sqrt{\varepsilon} M\int_{0}^{t}\Vert w(\cdot,s)\Vert_{\mathbb{L}^{2}_{\eta}}ds.
			\end{align*}
			Here we suppose $ U(\cdot;q(t)) $ and $ \nabla_{\eta}U $ is uniformly bounded by $ M $.
			$$ \max\{\sup_{t\in[0,T]}\Vert U(\cdot;q(t))\Vert_{\mathbb{L}^{\infty}_{\eta}},\sup_{t\in[0,T]}\Vert \nabla_{\eta}U(\eta;q(t)) \Vert_{\mathbb{L}^{\infty}_{\eta}}\}\le M <+\infty. $$
			By the Gronwall's inequality \Cref{lem:Grownwall}, we immediately come to the conclusion. 
		\end{proof}\par 
		The key idea of the following proof is to notice that, in our case
		\begin{align}
			\eta &= \sqrt{\frac{1}{2}}\left(  Q_h(t)\mathcal{A}^{\dagger}(t) + \overline{Q}_h(t) \mathcal{A}(t)\right),\quad \forall t\in [0,+\infty);\label{eqn:xop}\\
			p &= \sqrt{\frac{1}{2}}\left(  P_h(t)\mathcal{A}^{\dagger}(t) + \overline{P}_h(t) \mathcal{A}(t)\right),\quad \forall t\in [0,+\infty).\label{eqn:pop}
		\end{align}
		Here $ \eta $ denotes the multiplicative operator $ \eta:\mathcal{S}(\mathbb{R}^{d}_{\eta})\to (\mathcal{S}(\mathbb{R}^{d}_{\eta}))^{d} $ that maps any Schwartz function $ f $ to $ \{ \eta_{j}f \}_{j=1}^{d}\in (\mathcal{S}(\mathbb{R}^d_{\eta}))^{d} $. The $ p $ denotes the differential operator $ p:\mathcal{S}(\mathbb{R}_{\eta}^d)\to (\mathcal{S}(\mathbb{R}_{\eta}^d))^{d} $ that maps any Schwartz function $ g $ to $\{ -\iu\partial_{\eta_j} g \}_{j=1}^{d}\in (\mathcal{S}(\mathbb{R}_{\eta}^d))^{d} $.	These equalities are provided in formula (3.28) and (3.29) in \cite{Hagedorn1998}. \par 
		In the following part, $ U(\eta;q(t))w(\eta,t) $ is abbreviated as $ Uw(\eta,t) $.\par
		\subsection*{A.2\ Proof of estimate \eqref{eqn:err_H1norm} and \eqref{eqn:err_H1moment}}
		Here we only prove \eqref{eqn:err_H1norm}, and the estimate \eqref{eqn:err_H1moment} follows from the same steps as below.\par
		By the Duhamel's principle \eqref{eqn:Duhamel}, we have
		\begin{align}
			\nabla_\eta w &= \nabla_\eta R_{t,0}w_0(\eta) -\iu \sqrt{\varepsilon}\nabla_{\eta}\int_{0}^{t}R_{t,s}Uw(\eta,s)ds\notag\\
			&= \nabla_\eta R_{t,0}w_0(\eta) -\iu \sqrt{\varepsilon}\int_{0}^{t}R_{t,s}\nabla_{\eta}[Uw(\eta,s)]ds
			-\iu \sqrt{\varepsilon}\int_{0}^{t}[\nabla_{\eta}R_{t,s}-R_{t,s}\nabla_{\eta}]Uw(\eta,s)ds.\label{eqn:est_nabla}
		\end{align}
		Here $ R_{t,s}\nabla_{\eta}[Uw(\eta,s)] $ is interpreted as the operator $ R_{t,s} $ acting on each component of the vector $ \nabla_{\eta}[Uw(\eta,s)] $. \par 
		To deal with the first term in \eqref{eqn:est_nabla}, we notice that
		\begin{align}\label{eqn:QPFormula}
			Q_h(t) = \frac{1}{\sqrt{2}}\exp\left( 2\iu \int_{0}^{t}BB^{T}(s)ds \right),\quad P_{h}(t) = \sqrt{2}\iu\exp\left( 2\iu \int_{0}^{t}BB^{T}(s)ds \right).
		\end{align} 
		Then by simple observation that $ w_0(\eta) = \varepsilon^{-d/4}\varphi^{1}_{0}[0,0,2^{-1/2}E_d,\sqrt{2}\iu E_d](\eta) $, we have,
		\begin{align*}
			w(\eta,t) = R_{t,0}w_0 &= \varepsilon^{-d/4}\varphi^{1}_{0}[0,0,Q_h(t),P_h(t)](\eta)\\
			&=(\varepsilon\pi)^{-d/4}[\operatorname{det}Q_h(t)]^{-1/2}\exp\left( \frac{\iu}{2}\eta^{T}P_hQ_h^{-1}\eta \right)\\
			&=(\varepsilon\pi)^{-d/4}[\operatorname{det}Q_h(t)]^{-1/2}\exp\left( -\eta^{T}\eta \right).
		\end{align*}
		The last equality holds by the \eqref{eqn:QPFormula}. And we have $ \nabla_\eta R_{t,0}w_0(\eta) = \eu^{\iu\theta}\nabla_{\eta} w_0(\eta) $, where $ \theta $ is a real constant satisfying 
		$$ \eu^{\iu\theta} = [\det Q_{h}(0)/\det Q_{h}(t) ]^{1/2} .$$\par
		As for the second term in \eqref{eqn:est_nabla}, we have, by a simple calculation, 
		\begin{align*}
			-\iu \sqrt{\varepsilon}\int_{0}^{t}R_{t,s}\nabla_{\eta}[Uw(\eta,s)]ds &=- \iu\sqrt{\varepsilon}\bigg\{\int_{0}^{t}R_{t,s}[\nabla_{\eta}U(\eta;q(s))]w(\eta,s)ds \\&\quad+ \int_{0}^{t}R_{t,s}U(\eta;q(s))[\nabla_{\eta}w(\eta,s)]ds  \bigg\}.
		\end{align*}
		\par 
		The third term in \eqref{eqn:est_nabla} brings us the most difficulties. By \eqref{eqn:pop}, we substitute the differential operator by raising and lowering operators
		\begin{align}
			-\iu\left[\nabla_{\eta}R_{t,s} - R_{t,s}\nabla_{\eta}\right] &= \sqrt{\frac{1}{2}}\left\{ \left[ P_h(t)\mathcal{A}^{\dagger}(t) + \overline{P}_h(t) \mathcal{A}(t)\right] R_{t,s}  - R_{t,s}\left[ P_h(s)\mathcal{A}^{\dagger}(s) + \overline{P}_h(s) \mathcal{A}(s) \right]\right\}\notag\\
			&\begin{gathered}\label{eqn:comm_semi_p}	
				= \sqrt{\frac{1}{2}}\Big\{ \left[ P_h(t)\mathcal{A}^{\dagger}(t) + \overline{P}_h(t) \mathcal{A}(t)\right] R_{t,s}  - R_{t,s}\left[ P_h(t)\mathcal{A}^{\dagger}(s) + \overline{P}_h(t) \mathcal{A}(s) \right] \\
				\quad + R_{t,s}\left[ P_h(t)\mathcal{A}^{\dagger}(s) + \overline{P}_h(t) \mathcal{A}(s)\right]   - R_{t,s}\left[ P_h(s)\mathcal{A}^{\dagger}(s) + \overline{P}_h(s) \mathcal{A}(s) \right]\Big\}.\end{gathered}
		\end{align}
		Now we can represent $ Uw(\eta,s)\in\mathbb{L}^{2}(\mathbb{R}^d) $ as superposition of the Hagedorn's wave-packets. 
		$$ Uw(\eta,s) = \sum_{k\in\mathbb{N}^{d}}d_{k}(s)\varphi_{k}^{1}[0,0,Q_h(s),P_{h}(s)](\eta),\quad d_{k}(s)\in \mathbb{C}. $$
		By the definition of raising and lowering operators, we have
		\begin{align*}
			\left[ P_h(t)\mathcal{A}^{\dagger}(s) + \overline{P}_h(t) \mathcal{A}(s) \right]Uw(\eta,s) &= P_{h}(t)\bigg( \sum_{k\in\mathbb{N}^{d}}\sqrt{k_j+1}d_k(s)\varphi_{k+\langle j \rangle}^{1}[0,0,Q_h(s),P_{h}(s)]\bigg)_{j=1}^{d}\\
			&\quad + \overline{P}_{h}(t)\bigg( \sum_{k\in\mathbb{N}^{d}}\sqrt{k_j}d_k(s)\varphi_{k-\langle j \rangle}^{1}[0,0,Q_h(s),P_{h}(s)]\bigg)_{j=1}^{d}.
		\end{align*} 
		And the action of the non-autonomous semi-group operator $ R_{t,s} $ on each wave-packets is represented by
		$$ R_{t,s}:\varphi^{1}_{k}[0,0,Q_{h}(s),P_h(s)]\mapsto \varphi^{1}_{k}[0,0,Q_{h}(t),P_h(t)], \quad \forall k\in\mathbb{N}^{d}.  $$ 
		This leads to 
		\begin{align*}
			R_{t,s}\left[ P_h(t)\mathcal{A}^{\dagger}(s) + \overline{P}_h(t) \mathcal{A}(s) \right]Uw(\eta,s)&= P_{h}(t)\bigg( \sum_{k\in\mathbb{N}^{d}}\sqrt{k_j+1}d_k(s)\varphi_{k+\langle j \rangle}^{1}[0,0,Q_h(t),P_{h}(t)]\bigg)_{j=1}^{d}\\
			&\quad + \overline{P}_{h}(t)\bigg( \sum_{k\in\mathbb{N}^{d}}\sqrt{k_j}d_k(s)\varphi_{k-\langle j \rangle}^{1}[0,0,Q_h(t),P_{h}(t)]\bigg)_{j=1}^{d}.
		\end{align*}
		The right hand side equals to $ \left[ P_h(t)\mathcal{A}^{\dagger}(t) + \overline{P}_h(t) \mathcal{A}(t)\right] R_{t,s}Uw(\eta,s) $. Thus the first part in RHS of \eqref{eqn:comm_semi_p} vanishes. \par 
		To deal with the second part in RHS of \eqref{eqn:comm_semi_p}, we follow the same spirit of the above procedures 
		\begin{align*}
			\Big\{R_{t,s}\big[ &P_h(t)\mathcal{A}^{\dagger}(s) + \overline{P}_h(t) \mathcal{A}(s)\big] - R_{t,s}\left[ P_h(s)\mathcal{A}^{\dagger}(s) + \overline{P}_h(s) \mathcal{A}(s) \right]\Big\}Uw(\eta,s)\\ &=R_{t,s}\left[P_{h}(t)-P_{h}(s)\right]\Big( \sum_{k\in\mathbb{N}^{d}}\sqrt{k_j+1}d_k(s)\varphi_{k+\langle j \rangle}^{1}[0,0,Q_h(s),P_{h}(s)]\Big)_{j=1}^{d}\\
			&\quad +R_{t,s}\left[ \overline{P}_{h}(t)-\overline{P}_{h}(s)\right]\Big( \sum_{k\in\mathbb{N}^{d}}\sqrt{k_j}d_k(s)\varphi_{k-\langle j \rangle}^{1}[0,0,Q_h(s),P_{h}(s)]\Big)_{j=1}^{d}  \\
			&=R_{t,s}\left[P_{h}(t)-P_{h}(s)\right] \mathcal{A}^{\dagger}(s)Uw(\eta,s) + R_{t,s}\left[ \overline{P}_{h}(t)-\overline{P}_{h}(s)\right] \mathcal{A}(s)Uw(\eta,s).
		\end{align*}
		Notice that $ P_h(t) = P_{h}(t-s)P_{h}(s) $, we can rewrite the above formula into
		\begin{align*}
			R_{t,s}\left[P_{h}(t-s)-E_d\right]P_h(s) \mathcal{A}^{\dagger}(s)Uw(\eta,s) + R_{t,s}\left[ \overline{P}_{h}(t-s)-E_d\right]\overline{P}_h(s) \mathcal{A}(s)Uw(\eta,s).
		\end{align*}
		Separating the real part and the imaginary part of the matrix $ P_h(t-s) = P_{h,R}(t-s) + \iu P_{h,I}(t-s) $, we have
		\begin{align*}
			&R_{t,s}\left[P_{h}(t-s)-E_d\right]P_h(s) \mathcal{A}^{\dagger}(s)Uw(\eta,s) + R_{t,s}\left[ \overline{P}_{h}(t-s)-E_d\right]\overline{P}_h(s) \mathcal{A}(s)Uw(\eta,s)\\
			&= R_{t,s}\left[P_{h,R}(t-s)-E_d\right]P_h(s) \mathcal{A}^{\dagger}(s)Uw(\eta,s) + R_{t,s}\left[ P_{h,R}(t-s)-E_d\right]\overline{P}_h(s) \mathcal{A}(s)Uw(\eta,s)\\
			&\quad+\iu R_{t,s} P_{h,I}(t-s)P_h(s) \mathcal{A}^{\dagger}(s)Uw(\eta,s) -\iu R_{t,s} P_{h,I}(t-s)\overline{P}_h(s) \mathcal{A}(s)Uw(\eta,s).
		\end{align*}
		Here $ P_{h,R} $ and $ P_{h,I} $ are real matrices whose matrix 2-norm is less than $ 2 $ for all $ t,s\in \mathbb{R} $. 
		Recalling the general formula for $ Q_h(s) $ and $ P_{h}(s) $ given in \eqref{eqn:QPFormula}, we then have 
		\begin{align*}
			&R_{t,s}\left[P_{h}(t-s)-E_d\right]P_h(s) \mathcal{A}^{\dagger}(s)Uw(\eta,s) + R_{t,s}\left[ \overline{P}_{h}(t-s)-E_d\right]\overline{P}_h(s) \mathcal{A}(s)Uw(\eta,s)\\
			&= -\sqrt{2}R_{t,s}\left[P_{h}(t-s)-E_d\right]\iu\nabla_{\eta}[Uw(\eta,s)]\\
			&\quad +\iu R_{t,s} P_{h,I}(t-s)\sqrt{2}\iu\exp(\int_{0}^{s}BB^{T}(\tau)d\tau)\mathcal{A}^{\dagger}(s)Uw(\eta,s)  \\&\quad+\iu R_{t,s} P_{h,I}(t-s)\sqrt{2}\iu\exp(\int_0^{s}BB^{T}(\tau)d\tau) \mathcal{A}(s)Uw(\eta,s)\\
			&=-\sqrt{2}R_{t,s}\left[P_{h,R}(t-s)-E_d\right] \iu\nabla_{\eta}[Uw(\eta,s)] -2R_{t,s}P_{h,I}(t-s)[\eta Uw(\eta,s)].
		\end{align*}
		We apply the representation formula of operator $ \eta $ \eqref{eqn:xop} in the last equality. \par 
		Thus, by the Cauchy's inequality and the integral form of Minkowski's inequality, we have
		\begin{align}
			\Vert \nabla_{\eta} w \Vert_{\mathbb{L}^{2}_{\eta}} &\le 5\bigg\{ \Vert \nabla_{\eta} w_0 \Vert_{\mathbb{L}^{2}_{\eta}} + \sqrt{\varepsilon} M \int_{0}^{t}\big(\Vert w(\cdot,s) \Vert_{\mathbb{L}^{2}_{\eta}} + \Vert \nabla_{\eta} w \Vert_{\mathbb{L}^{2}_{\eta}}\big) ds\notag \\
			&\quad+ 3\sqrt{2} M\sqrt{\varepsilon}\int_{0}^{t}[\Vert w(\cdot,s) \Vert_{\mathbb{L}^{2}_{\eta}} + \Vert \nabla_{\eta} w(\cdot,s) \Vert_{\mathbb{L}^{2}_{\eta}}] ds+4M\sqrt{\varepsilon} \int_{0}^{t}\Vert \eta w(\eta,s)  \Vert_{\mathbb{L}^{2}_{\eta}}ds\bigg\}\notag\\
			&\le C_1\bigg\{ \Vert \nabla_{\eta} w_0 \Vert_{\mathbb{L}^{2}_{\eta}} + \sqrt{\varepsilon}\int_{0}^{t}[\Vert w(\cdot,s) \Vert_{\mathbb{L}^{2}_{\eta}} + \Vert \eta w(\eta,s)  \Vert_{\mathbb{L}^{2}_{\eta}} + \Vert \nabla_{\eta} w(\cdot,s) \Vert_{\mathbb{L}^{2}_{\eta}}] ds\bigg\}.\label{eqn:estiNablaw}
		\end{align}	
		Here 
		$$ \Vert \eta w(\eta,s) \Vert_{\mathbb{L}^{2}_{\eta}}^{2}\triangleq \int_{0}^{t}(\eta^{T}\eta)|w|^2(\eta,s) d\eta, $$
		and $ C_1 $ is a constant depending on $ T $.\par 
		Now we have to establish an estimate for $ \Vert \eta w(\eta,s) \Vert_{\mathbb{L}^{2}_{\eta}}^{2} $. Again by Duhamel's principle, we have
		\begin{align}\label{eqn:weta}
			\eta w(\eta,t) = \eta R_{t,0}w_{0}(\eta) -\iu\sqrt{\varepsilon}\eta\int_{0}^{t}R_{t,s}Uw(\eta,s)ds.
		\end{align}
		Following the same procedure, we can show $ \eta R_{t,0}w_{0}(\eta)  = \eu^{\iu\theta}\eta w_{0}(\eta)$, where $ \theta $ is a real constant defined before. Apply \eqref{eqn:xop} to the second part of RHS in \eqref{eqn:weta}, we obtain 
		\begin{align*}
			-\iu\sqrt{\varepsilon}\eta\int_{0}^{t}R_{t,s}Uw(\eta,s)ds = -\iu\sqrt{\frac{\varepsilon}{2}}\int_{0}^{t}[Q_{h}(t)\mathcal{A}^{\dagger}(t) + \overline{Q}_{h}(t)\mathcal{A}(t)]R_{t,s}Uw(\eta,s)ds.	
		\end{align*}
		Previous procedures manifest that 
		\begin{align*}
			[Q_{h}(t)\mathcal{A}^{\dagger}(t) &+ \overline{Q}_{h}(t)\mathcal{A}(t)]R_{t,s} =R_{t,s}[Q_{h}(t)\mathcal{A}^{\dagger}(s) + \overline{Q}_{h}(t)\mathcal{A}(s)]\\
			&= R_{t,s}[Q_{h}(s)\mathcal{A}^{\dagger}(s) + \overline{Q}_{h}(s)\mathcal{A}(s)]\\
			&\quad + \{[Q_{h}(t)-Q_{h}(s)]\mathcal{A}^{\dagger}(s) + [\overline{Q}_{h}(t)-\overline{Q}_{h}(s)]\mathcal{A}(s)\},
		\end{align*}
		which implies
		\begin{align*} 
			-\iu\sqrt{\varepsilon}\eta\int_{0}^{t}R_{t,s}Uw(\eta,s)ds&=-\iu\sqrt{\frac{\varepsilon}{2}}\int_{0}^{t}R_{t,s}[Q_{h}(s)\mathcal{A}^{\dagger}(s) + \overline{Q}_{h}(s)\mathcal{A}(s)]Uw(\eta,s)ds\notag\\
			&\quad-\iu\sqrt{\frac{\varepsilon}{2}}\int_{0}^{t}R_{t,s}[Q_{h}(t)-Q_{h}(s)]\mathcal{A}^{\dagger}(s) Uw(\eta,s)ds\\
			&\quad-\iu\sqrt{\frac{\varepsilon}{2}}\int_{0}^{t} R_{t,s}[\overline{Q}_{h}(t)-\overline{Q}_{h}(s)]\mathcal{A}(s)Uw(\eta,s)ds.
		\end{align*}
		By the obvious relation $ Q_{h}(t) = Q_{h}(t-s)Q_{h}(s) $, we have 
		\begin{align*}
			&\Big\{[Q_{h}(t)-Q_{h}(s)]\mathcal{A}^{\dagger}(s) + [\overline{Q}_{h}(t)-\overline{Q}_{h}(s)]\mathcal{A}(s)\Big\}Uw(\eta,s)\\
			&= \Big\{[Q_{h}(t-s)-E_d]Q_{h}(s)\mathcal{A}^{\dagger}(s) + [\overline{Q}_{h}(t-s)-E_d]\overline{Q}_{h}(s)\mathcal{A}(s)\Big\}Uw(\eta,s)\\
			&=\Big\{[Q_{h,R}(t-s)-E_d]Q_{h}(s)\mathcal{A}^{\dagger}(s) + [Q_{h,R}(t-s)-E_d]\overline{Q}_{h}(s)\mathcal{A}(s)\Big\}Uw(\eta,s)\\
			&\quad + \Big\{\iu Q_{h,I}(t-s)Q_{h}(s)\mathcal{A}^{\dagger}(s) -\iu Q_{h,I}(t-s)\overline{Q}_{h}(s)\mathcal{A}(s)\Big\}Uw(\eta,s),
		\end{align*}
		where $ Q_{h}(t-s) = Q_{h,R} + \iu Q_{h,I} $. Both $ Q_{h,R} $ and $ Q_{h,I} $ are real matrices whose matrix 2-norm is less that $ 2 $.
		Again by the explicit formula for $ Q_h(s) $ and \eqref{eqn:pop}, we have 
		\begin{align*}
			&\Big\{\iu Q_{h,I}(t-s)Q_{h}(s)\mathcal{A}^{\dagger}(s) -\iu Q_{h,I}(t-s)\overline{Q}_{h}(s)\mathcal{A}(s)\Big\}Uw(\eta,s) \\
			& =\iu \sqrt{\frac{1}{2}}Q_{h,I}(t-s)\Big\{ \exp(2\iu \int_{0}^{s}BB^{T}(\tau)d\tau) \mathcal{A}^{\dagger}(s)-\exp(-2\iu \int_{0}^{s}BB^{T}(\tau)d\tau)\mathcal{A}(s)\Big\}Uw(\eta,s)\\
			&=\frac{1}{2}Q_{h,I}(t-s)\Big\{ P_{h}(s) \mathcal{A}^{\dagger}(s)+\overline{P}_{h}(s)\mathcal{A}(s)\Big\}Uw(\eta,s)\\
			&=-\frac{\iu}{2}Q_{h,I}(t-s) \nabla_{\eta}[Uw(\eta,s)].
		\end{align*}
		The arguments above give a similar estimate for $ \Vert \eta w\Vert_{\mathbb{L}^{2}_{\eta}} $.
		\begin{align}\label{eqn:estiEtaw}
			\Vert \eta w(\eta,t)\Vert_{\mathbb{L}^{2}_{\eta}} \le C_2\bigg\{ \Vert \eta w_0  \Vert_{\mathbb{L}^{2}_{\eta}} +\sqrt{\varepsilon}\int_{0}^{t}[\Vert w(\cdot,s) \Vert_{\mathbb{L}^{2}_{\eta}} + \Vert \eta w(\eta,s)  \Vert_{\mathbb{L}^{2}_{\eta}} + \Vert \nabla_{\eta} w(\cdot,s) \Vert_{\mathbb{L}^{2}_{\eta}}] ds\bigg\}
		\end{align}
		Combining \eqref{eqn:estiNablaw} and \eqref{eqn:estiEtaw} and the Gronwall's inequality, we conclude 
		\begin{align}
			\Vert \eta w\Vert_{\mathbb{L}^{2}_{\eta}}+ \Vert\nabla_{\eta} 
			w(\cdot,t) \Vert_{\mathbb{L}^{2}_{\eta}} \le C_{3}(T)\varepsilon^{-d/4},
		\end{align}
		which gives the desired estimate \eqref{eqn:err_H1norm}.
	
\end{document}